\documentclass[reqno,12pt,oneside]{amsart}
\usepackage{amssymb}
\usepackage{amsmath}
\usepackage{amsthm}
\usepackage{mathtools}
\usepackage[svgnames]{xcolor}
\usepackage[pdftex]{hyperref}
\usepackage{enumitem}
\usepackage{booktabs}
\usepackage{placeins }
\usepackage{soul}
\usepackage{enumitem}
\usepackage{verbatim}

\usepackage{tikz}
\usepackage{tkz-tab}
\usepackage{tkz-fct} 
\usetikzlibrary{calc,intersections}

\definecolor{color1}{RGB}{27,158,119}
\definecolor{color2}{RGB}{217,95,2}
\definecolor{color3}{RGB}{117,112,179}
\definecolor{color4}{RGB}{231,41,138}

\graphicspath{{./figures/}}

\usepackage{vmargin}%
\setmarginsrb{2cm}{2cm}{2cm}{2cm} {1cm}{1cm}{0.5cm}{0.5cm}

\hypersetup{
  pdftitle={},
  pdfauthor={Perla Kfoury, Stefan Le Coz},
  pdfsubject={},
  pdfkeywords={}, 
  colorlinks=true,
  linkcolor=DarkBlue , %
  citecolor=DarkRed, %
  filecolor=DarkMagenta, %
  urlcolor=DarkGreen, %
}

\newtheorem{theorem}{Theorem}[section]
\newtheorem{main}[theorem]{Main Results}
\newtheorem{proposition}[theorem]{Proposition}
\newtheorem{lemma}[theorem]{Lemma}

\theoremstyle{remark}
\newtheorem{remark}[theorem]{Remark}

\newcommand{\N}{\mathbb{N}}
\newcommand{\R}{\mathbb{R}}
\newcommand{\C}{\mathbb{C}}

\DeclareMathOperator{\sech}{sech}
\DeclareMathOperator{\cn}{cn}
\DeclareMathOperator{\dn}{dn}
\DeclareMathOperator{\sn}{sn}

\renewcommand{\leq}{\leqslant}
\renewcommand{\geq}{\geqslant}

\DeclareMathAlphabet{\mathpzc}{OT1}{pzc}{m}{it}
\renewcommand{\Re}{\mathcal R\!\mathpzc{e}}
\renewcommand{\Im}{\mathcal I\!\mathpzc{m}}

\usepackage{color}

\usepackage{placeins } %
\usepackage{tikz}
\usepackage{tkz-tab}
\usepackage{tkz-fct} 
\usetikzlibrary{calc,intersections}
\usepackage{pgfplots}
\usetikzlibrary{calc}
\usetikzlibrary{patterns}

\usepgfplotslibrary{colormaps}

\begin{document}

\title[Analysis of the quasi-periodic waves of NLS]{Analysis of quasi-periodic waves of cubic nonlinear Schr\"odinger equations}

\author[P.~Kfoury]{Perla Kfoury}
\address[Perla Kfoury]{Mathematics Division, School of Arts and Sciences,\newline\indent American University, P.O.Box 28282, Dubai, UAE }
\email[Perla Kfoury]{pkfoury@aud.edu}

\author[S.~Le Coz]{Stefan Le Coz}
\thanks{The work of P. K. and S. L. C. is 
  partially supported by ANR-11-LABX-0040-CIMI and the ANR project NQG ANR-23-CE40-0005}
\address[Stefan Le Coz]{
  \newline\indent
  Universit\'e de Toulouse ; INSA ; CNRS; IMT
  \newline\indent
  UPS IMT, F-31062 Toulouse Cedex 9,
  \newline\indent
  France}
\email[Stefan Le Coz]{slecoz@math.univ-toulouse.fr}

\author[T.-P.~Tsai]{Tai-Peng Tsai}
\thanks{The work of T.-P. T. is partially supported by the NSERC grant RGPIN-2023-04534.}
\address[Tai-Peng Tsai]{
Department of Mathematics,
\newline\indent
University of British Columbia,
\newline\indent
Vancouver BC
\newline\indent
Canada V6T 1Z2}
\email[Tai-Peng Tsai]{ttsai@math.ubc.ca}

\subjclass[2010]{35Q55 (35B10, 35A15)}

\date{October 17, 2025}
\keywords{nonlinear Schr\"odinger equation, standing waves, normalized gradient flow, variational method, quasi-periodic solutions}

\begin{abstract}
  We study the quasi-periodic standing wave solutions of the focusing and defocusing cubic nonlinear Schr\"odinger equations in dimension one. In the defocusing case, we establish  a diffeomorphic correspondence between the invariants of the ordinary differential equation of the wave profiles and the conserved quantities of the evolution equation. We introduce a numerical scheme to compute the minimizers of the energy at fixed mass and momentum for both focusing and defocusing cases. The scheme is based on a gradient flow approach with discrete renormalization at each time step.  The novelty of our scheme is that the renormalization step deals at the same time with the mass and the momentum constraints. In numerical experiments, we observe that a given solution of the profile ordinary differential equation is also a minimizer of the energy at corresponding mass and momentum. 
\end{abstract}
\maketitle

\section{Introduction}
We consider the one dimensional cubic nonlinear Schr\"odinger equation
\begin{equation} \label{eq:nls}
i \psi _t + \psi _{xx} + b | \psi |^2 \psi =0,
\end{equation}
where $ \psi : \R_t \times \R_x \rightarrow \C $ and $b \in \R \setminus \{0\}$. Equation~\eqref{eq:nls} appears in many areas of physics such as quantum mechanics, optics, and water waves, and serves as a model for nonlinear dispersive wave phenomena, more generally see \cite{Fi15, SuSu99}. It is said to be focusing if $b>0$, where the nonlinearity is attractive and defocusing if $ b<0$, where the nonlinearity is repulsive.

We are particularly interested in the spatially quasi-periodic solutions. They are solutions of~\eqref{eq:nls} for which there exist a period $T>0$ and a Floquet multiplier $\theta\in \R$ such that for any $x\in\mathbb R$  we have 
  \begin{equation}
\psi(\cdot,x+T)=e^{i\theta}\psi(\cdot ,x).\label{eq:quasi}
\end{equation}
Note that $\theta$ is the increment of the phase over a period.
Observe that the space of quasi-periodic functions is left invariant by~\eqref{eq:nls}, i.e.~if the initial data verifies~\eqref{eq:quasi}, then the corresponding solution of~\eqref{eq:nls} (if it exists) also verifies~\eqref{eq:quasi} for the same parameters. 
Given a Floquet multiplier $\theta\in\mathbb R$, we define the corresponding space of quasi-periodic locally $H^1$ functions by
\begin{equation*}
H^{\theta}_T=\{f \in H^1_{loc}(\R):  f(x+T)=e^{i \theta} f(x), \forall x \in \R \}.
\end{equation*}
When the Floquet multiplier is $0$ we are in the case of periodic waves ($f(x+T)= f(x)$) and when it is $\pi$ we are in the case of anti-periodic solutions (i.e.~$f(x+T)=- f(x)$).

At least formally, the flow of~\eqref{eq:nls} in the quasi-periodic spaces $H^{\theta}_T$ conserves the mass $M$, the momentum $P$, and the energy $\mathcal{E}$, which are defined by
\begin{equation*}
M(\psi )= \frac{1}{2}\int_0^T | \psi|^2 dx, \quad 
P(\psi)= \frac{1}{2} \Im \int_0^T \psi \overline{\psi}_x dx,
\quad
\mathcal{E}(\psi)= \frac{1}{2} \int_0^T | \psi_x|^2 dx-\frac{b}{4} \int_0^T | \psi|^4 dx.
\end{equation*}
The simplest non-trivial solutions of~\eqref{eq:nls} are the standing waves, which have the form
\begin{equation*}
\psi (t,x)= e^{-iat} u(x), \quad a \in \R.
\end{equation*}
The profile function $u$ satisfies the ordinary differential equation
\begin{equation} \label{eq:ode}
u_{xx}+au+b |u|^2 u=0.
\end{equation}
This ordinary differential equation might be interpreted as a first order Hamiltonian system for the (complex-valued) variables $(u,u_x)$. 
The conserved quantities of~\eqref{eq:ode} on $\C$ are the angular momentum $J$ and the ordinary differential equation energy $E$, defined by
\begin{equation}
\label{eq:def-J-E}
J= \Im (u \overline{u}_x),\quad E= \frac{1}{2} |u_x |^2 + \frac{a}{2}
|u|^2 + \frac{b}{4} |u|^4 .
\end{equation}
These two quantities might be inferred from Noether's principle or
simply obtained by hand (see Section~\ref{sec:ode}). It is well known (see e.g.~\cite{GaHa07-2,GaHa07-1} and the references therein) that the bounded complex valued solutions of the ordinary differential equation~\eqref{eq:ode} can be classified depending on the values of $J$ and $E$ inside a set of admissible parameters (see Section~\ref{sec:ode} for details). It is also well known that bounded, non-constant, real-valued solutions of~\eqref{eq:ode} will be periodic or anti-periodic functions given by Jacobi elliptic functions: dnoidal ($\dn$), cnoidal ($\cn$) and snoidal ($\sn$), see Section~\ref{sec:Jacobi-func} for details.

A variational characterization of real-valued solutions to~\eqref{eq:ode} was established in \cite{GuLeTs17} (see also \cite{KfLe24} for general nonlinearities). The dnoidal, cnoidal and snoidal Jacobi elliptic functions were recovered as solutions of a minimization problem for the energy at fixed mass in periodic or anti-periodic function spaces with prescribed values of the period and mass.

In this article, we will consider the following minimization problem:
\begin{equation} \label{eq:minim-problem}
  \min \{ \mathcal{E}(u): u \in H^{\theta}_T,\, M(u)=m, P(u)=p\},\quad m>0,\,p\in\mathbb R.
\end{equation}
Our main goal is to characterize variationally the complex-valued solutions of the ordinary differential equation~\eqref{eq:ode} by  identifying them as the minimizers of the minimization problem~\eqref{eq:minim-problem} for a suitable choice of parameters, to be defined rigorously later. This will be achieved using a mixture of analytical and numerical methods. Our main results can be summarized as follows.

\begin{main}
  \label{main}
  Let $b\in \mathbb R\setminus\{0\}$, $a\in\mathbb R$ and $(J,E)\in \R^2$ be in the admissible domain for the existence of a quasi-periodic solution to~\eqref{eq:ode}. Then there exist $T>0$, $\theta\in[0,2\pi)$, $m>0$, $p\in\R$ such that the minimization problem~\eqref{eq:minim-problem} has a numerical solution which is the solution of~\eqref{eq:ode} with parameters $(J,E)$. If $b<0$, then the correspondence $(J,E)\to(m,p)$ is  a diffeomorphism.
\end{main}

We now describe the main ideas and methods.

 Since our goal is to establish a link between the solutions of an ordinary differential equation and the minimizers of a variational problem, we first connect the invariants of the ordinary differential equation with the parameters of the variational problem. In particular, in the defocusing case, we establish a diffeomorphic correspondence between the conserved quantities of the ordinary differential equation~\eqref{eq:ode} and those of the nonlinear Schrödinger equation~\eqref{eq:nls}. Our approach relies in part on the work of Gallay and Haragus \cite{GaHa07-2}.

The second step is to connect the minimizers with the solutions of the equation. To this aim, we develop a numerical scheme for the computations of the minimizers. Our scheme is  a gradient-flow method with discrete normalization. It is used to find the minimizer of the energy $\mathcal{E}$ with fixed mass $m>0$ and fixed momentum $p\in\mathbb R$. At each time step, we evolve in the direction of the energy gradient and then renormalize the mass and momentum of the resulting function. Similar schemes (for single-constraint problems) are well known in the physics literature under the name “imaginary time method”. When there is no momentum and only real-valued functions are considered, such an approach to compute minimizers was developed by Bao and Du \cite{BaDu04}. For the nonlinear Schrödinger equation on the line $\R$ with focusing cubic nonlinearity, Faou and Jezequel \cite{FaJe18} provided a theoretical analysis of the various levels of discretization of this method, from the continuous formulation to the fully discrete scheme. To characterize variationally the Jacobi elliptic functions, Gustafson, Le Coz, and Tsai \cite{GuLeTs17} developed a numerical method to obtain minimizers of the energy with fixed mass $m>0$ and zero momentum. A related study was carried out by Besse, Duboscq, and Le Coz \cite{BeDuLe20} for the gradient flow on nonlinear quantum graphs. The main originality of our approach is that we are able to handle renormalization on two constraints instead of one. This is achieved by the introduction of a suitable auxiliary evolution problem. The validity of our approach is supported by numerical experiments. For various values of $(J,E)$ we compute the solutions of the ordinary differential equation together with their period, Floquet multiplier, mass, and momentum. We then run the normalized gradient-flow algorithm with these parameters and observe that the computed minimizer coincides, up to numerical errors, with the corresponding solution of the ordinary differential equation.

The rest of the paper is organized as follows.

In Section~\ref{sec:ode-min}, we start from the ordinary differential equation~\eqref{eq:ode} and examine the links between its solutions and those of the minimization problem~\eqref{eq:minim-problem}. We begin in Section~\ref{sec:Jacobi-func} by reviewing the well-known properties of Jacobi elliptic functions. In Sections~\ref{sec:analysis-of-the-ode} and~\ref{sec:domains}, we analyze the existence regions for the ordinary differential equation~\eqref{eq:ode} and recall the general form of its solutions. In Section~\ref{sec:period}, we provide a formula for the period function. In Section~\ref{sec:Mass-Momentum}, we derive the expressions of the mass $M$ and the momentum $P$ in terms of $J$ and $E$, both within the admissible domain for $(J,E)$ and on its boundaries. In Section~\ref{sec:diffeo}, we show that the mapping  $(J,E)\mapsto(\tilde M(J,E),\tilde P(J,E))$ is a diffeomorphism in the defocusing case.

In Section~\ref{sec:min-ode}, we then turn to the solutions of the minimization problem~\eqref{eq:minim-problem} and study the links between its minimizers and the solutions of the ordinary differential equation~\eqref{eq:ode}. We begin Section~\ref{sec:variational-problems} by recalling the results of \cite{GuLeTs17} concerning global variational characterizations of elliptic-function periodic waves as constrained-mass energy minimizers among periodic functions and suitable subspaces of such functions. In Section~\ref{sec:ode}, we explain how we obtain numerically the solutions of the ordinary differential equation~\eqref{eq:ode}. In Section~\ref{sec:gradient-flow}, we present the gradient-flow method with discrete normalization, developed to tackle minimization problems involving simultaneous renormalization of the mass and the momentum. The space-time discretization is presented in Section~\ref{sec:discretization}. Finally, in Section~\ref{subsec:experiments}, we present several numerical experiments based on this scheme.

\section{From the ordinary differential equation to the minimization problem} \label{sec:ode-min}
\subsection{Jacobi elliptic functions} \label{sec:Jacobi-func}
The Jacobi elliptic functions are standard forms of elliptic functions. We give in this section some definitions and a few relevant properties and refer to e.g.~\cite{La89} for more information. The three basic functions are denoted $\cn(x,k)$, $\dn(x,k)$, and $\sn(x,k)$, where $k \in (0,1)$ is known as the elliptic modulus.

The incomplete elliptic integral of the first kind is defined by
\begin{equation*}
x=\mathbf{F}(\phi,k):=\int_0^\phi\frac{d\theta}{\sqrt{1-k^2\sin^2(\theta)}},
\end{equation*}
where $\phi$ is called the Jacobi amplitude, and the Jacobi elliptic functions are defined through the inverse of $\mathbf{F}(.,k)$:
\begin{equation*}
\sn(x,k):=\sin (\phi), \quad \cn (x,k):=\cos(\phi), \quad \dn(x,k):=\sqrt{1-k^2\sin^2(\phi)}.
\end{equation*}
The relations 
\begin{equation*}
1=\sn^2+\cn^2=k^2\sn^2+\dn^2
\end{equation*}
follow. The period of the elliptic functions can be expressed in terms of the complete elliptic integral of the first kind
\begin{equation*}
\mathbf  K(k):=\mathbf  F \left( \frac{\pi}{2},k \right), \quad \mathbf K(k) \rightarrow \left\{
\begin{array}{ll}
\frac{\pi}{2}, & k \rightarrow 0, \\
\infty, & k \rightarrow 1.
\end{array}
\right.
\end{equation*}
The functions $\cn$ and $\sn$ have fundamental period $4\mathbf K$, while the function $\dn$ has a fundamental period of $2\mathbf K$.

For $k \in (0,1)$, the incomplete elliptic integral of the second kind is defined by:
\begin{equation*}
\mathbf{E}(\phi,k):=\int_0^\phi \sqrt{1-k^2 \sin^2(\theta)}d\theta.
\end{equation*}
The complete elliptic integral of the second kind is defined by
\begin{equation*}
\mathbf{E}(k):=\mathbf{E} \left(\frac{\pi}{2},k \right), \quad \mathbf{E}(0)= \frac{\pi}{2}, \quad \mathbf{E}(1)=1.
\end{equation*}

The derivatives of elliptic functions (with respect to $x$) can be expressed in terms of elliptic functions. For fixed $k \in (0,1)$, we have:
\begin{equation*}
\partial_x \sn=\cn.\dn, \quad \partial_x \cn= -\sn.\dn, \quad \partial_x \dn=-k^2\cn.\sn.
\end{equation*}
We can easily verify that $\sn$, $\cn$ and $\dn$ are solutions of~\eqref{eq:ode}
with specific values of the coefficients $a,b \in \R$ given by:
\begin{equation*}
a=1+k^2, \quad b=-2k^2, \quad \text{for } u=\sn,
\end{equation*}
\begin{equation*}
a=1-2k^2, \quad b=2k^2, \quad \text{for } u=\cn,
\end{equation*}
\begin{equation*}
a=-(2-k^2), \quad b=2, \quad \text{for } u=\dn.
\end{equation*}

\subsection{Analysis of the profile equation}
\label{sec:analysis-of-the-ode}

In this section, we study the bounded solutions of the ordinary differential equation~\eqref{eq:ode}. We will denote by $u$ a solution of~\eqref{eq:ode} with parameters $J$ and $E$.
Recall that the invariants $J$ and $E$ defined in~\eqref{eq:def-J-E} may be obtained in the following way. 
To find the angular momentum $J$, we multiply~\eqref{eq:ode} by $ \overline{u} $, and we take the imaginary part to obtain $ \Im ( u_{xx} \overline{u})=0$. Since $ \partial_x (u_x \overline{u} )= u_{xx} \overline{u}+ |u_x|^2 $, we have $\partial_x ( \Im ( u_{x} \overline{u}))=0 $.
Therefore there exists $J \in \R$ such that $ \Im ( u \overline{u}_{x}) \equiv J$. To find the energy $E$ we multiply~\eqref{eq:ode} by $ \overline{u}_x $ and we take the real part.

We distinguish between two different cases depending on whether or not $J = 0$ and in each case we study the defocusing and focusing cases. The analysis presented here gathers elements already present in earlier works such as \cite{GaHa07-2,GuLeTs17}. 

\subsubsection{The case  $J \neq 0$}
We start with the case $J \neq 0  $. By definition of $J$, this implies in particular that $u(x) \neq 0$ for all $x \in \R $. Hence we can introduce the polar coordinates $u(x)= r(x) e^{i \phi (x)} $, with $r>0$ and $r , \phi \in C^2(\R)$. Since $u$ verifies~\eqref{eq:ode}, $(r,\phi)$ verifies
\begin{equation*}
\begin{cases}
  2\phi_xr_x+\phi_{xx}r=0,\\
  r_{xx}-\phi_x^2r+ar+br^3=0.
\end{cases}
\end{equation*}
Moreover $J= -r^2\phi _x \neq 0$ and this  implies that $\phi_x=-J/r^2 \neq 0$. Hence $u(x) \in \C$ with a non-trivial phase. Moreover, the system verified by $(r,\phi)$ can be rewritten as
\begin{equation*}
  \begin{cases}
  -\phi_xr^2=J,\\
  r_{xx}-\frac{J^2}{r^3}+ar+br^3=0.
\end{cases}
  \end{equation*}
The energy becomes
\begin{equation}
  \label{eq:def-E-V_J}
E=\frac{r_x^2}{2}+V_J(r),\quad V_J(r) = \frac{J^2}{2r^2} + a\frac{r^2}{2}+b \frac{r^4}{4}.
\end{equation}
We will call $V_J(r)$ the potential.

We start with the defocusing case, i.e.~we assume $b<0$. We will distinguish between three different cases depending on the values of $J$. The first case is when $ J^2 > \frac{4}{27}\frac{a^3}{b^2}$.  Then $V'_J(r) <0 $ for all $ r>0$, hence~\eqref{eq:ode} has no bounded solution in this case.

The second case is when $0< J^2 <  \frac{4}{27}\frac{a^3}{b^2}$ (which implies in particular $a>0$) then we can parametrize $J$  as
\begin{equation*}
J=q \frac{(q^2-a)}{b} = Q \frac{(Q^2-a)}{b} , \quad  \text{where } 0<q^2 < \frac{a}{3}< Q^2 < a ,
\end{equation*}
and $q$ and $Q$ have the same sign as $J$.
With this parametrization $V_J(r) $ has a unique local minimum at $r_Q=\sqrt{\frac{Q^2-a}{b}}$ and a unique local maximum at $ r_q= \sqrt{\frac{q^2-a}{b}}$. We define:
\begin{equation}
  \label{eq:def_E_-}
E_{-} (J)= V_J( r_Q)=\frac{1}{4b} (Q^2-a)(3Q^2+a), \quad E_{+} (J)= V_J(r_q)=\frac{1}{4b}(q^2-a) (3q^2+a).
\end{equation}
The curves $E_+$ and $E_-$ delimit the region of the $(J,E)$ plane where there exist bounded solutions to~\eqref{eq:ode} (see Figure~\ref{fig:E-as-function-of-J}). We have the following description.
\begin{itemize}
\item If $E= E_- (J) $ then $u$ is a plane wave, i.e.~$u(x)= r_Q e^{-iQx} $ up to phase shift.
\item If $ E=E_+ (J) $ then $u$ is a plane wave, i.e.~$u(x)=r_q e^{-iqx} $ up to phase shift, or $|u|$ is a solution homoclinic to $r_q$
  (see \cite{GaHa07-2} for the explicit formula).
\item If $E_-(J) < E < E_+ (J) $ the modulus $ r=|u|$ and the phase derivative $\phi _ x $ are periodic with the same period. If we denote by $r_1<r_2<r_3$ the three positive roots of $E- V_J(r)$, the (minimal) period is 
    \begin{equation}
\label{eq:period-T}
T(J,E)= 2 \int_{r_1} ^{r_2}\frac{{\rm d}{r}}{\sqrt{2(E-V_J(r)})}.
\end{equation}
    \end{itemize}
Finally the last case is when $J^2= \frac{4}{27}\frac{a^3}{b^2}$. In this case, the  potential $V_J(r)$ is strictly decreasing over $ \R_{+} $ with an inflection point at $r= \sqrt{\frac{-2a}{3b}}$. Since $u$ is bounded we have $E_{-}(J)=E_{+}(J) = \frac{-a^2}{3b}$, hence $u(x)= r e^{-iqx}$ (up to phase shift) with $q=\sqrt{a/3}$\,\,sign\,$J$.

\medskip

Now for the focusing case $b>0$ with $a>0$ ($a\leq 0$ respectively): the potential $V_J(r)$ is strictly convex for any $J \in \R$, and we can parametrize $J$ in a unique way as $ J=Q \frac{(Q^2-a)}{b}$, where $Q \in \R $, $Q^2 \geq a$, %
 ($Q \in \R$ respectively). Then $ V_J(r)$ has a unique critical point at $ r_Q=\sqrt{\frac{Q^2-a}{b}}$, where $V_J$ attains its global minimum. Consider $E_-(J)$ defined by the same expression as in~\eqref{eq:def_E_-}.
As in the defocusing case, we have the following description for solutions.
\begin{itemize}
\item If $E= E_- (J) $ then $u$ is a plane wave, i.e.~$u(x)= r_Q e^{-iQx} $ up to phase shift.
\item If $E_-(J) < E $ the modulus $ r=|u|$ and the phase derivative $\phi _ x $ are periodic with the same period. If we denote by $r_1<r_2$ the two positive roots of $E- V_J(r)$, the (minimal) period is also given by expression~\eqref{eq:period-T}. 
\end{itemize}

\subsubsection{The case $J=0$}
In the case $J=0$, $u(x) \in \R$ up to a constant phase, so we consider real-valued solutions, which we describe up to translation and sign change. 
The only real-valued, bounded solutions of~\eqref{eq:ode} are the elliptic functions, which are periodic (expect in limit cases). We have the following results (see e.g.~\cite{GuLeTs17}).

The  energy is given by
\begin{equation*}
E = \frac{u_x^2}{2}+V_0(u),\quad V_0(u)=a\frac{u^2}{2}+b\frac{u^4}{4}.
\end{equation*}
We start with the defocusing case $b<0$. Bounded solutions exist if and only if $a>0$, for $0\leq E\leq -\frac{a^2}{4b}$.
\begin{itemize}
\item When $E=0$, the only solution is the constant $0$.
  \item When
  $E=-\frac{a^2}{4b}$, the solutions are either the constant
  $\sqrt{\frac{-a}{b}}$ or the heteroclinic orbit
  $u(x)=\sqrt{\frac{-a}{b}}\tanh\left(x\sqrt{\frac a2}\right)$.
  \item When
    $0< E< -\frac{a^2}{4b}$, the solutions are of the form
    \begin{equation}
\label{eq:2}
u(x)=\frac{1}{\alpha} \sn \left(\mathbf{K}(k)+\frac{x}{\beta},k\right),
    \end{equation}
    where $k\in(0,1)$ and $b\beta^2=-2k^2\alpha^2$.
    \end{itemize}

    We now consider the focusing  case $b>0$.  Bounded solutions exist for any $a\in\R$ if  and only if $E\geq0$ when $a\geq0$ or $E\geq-\frac{a^2}{4b}$ when $a<0$. We first consider the case $a\geq0$.
    \begin{itemize}
    \item When $E=0$, the only solution is the constant $0$.
    \item When $E>0$, the solutions are of the form
\begin{equation}
  \label{eq:3}
u(x)=\frac{1}{\alpha} \cn \left(\frac{x}{\beta},k\right),
\end{equation}
where $k\in\left(0,\frac{1}{\sqrt{2}}\right]$ and  $b\beta^2=2k^2\alpha^2$. 
\end{itemize}
We now consider the case $a<0$.
    \begin{itemize}
    \item When $E=-\frac{a^2}{4b}$, the only solution is the constant $\sqrt{\frac{-a}{b}}$.
    \item When $-\frac{a^2}{4b}<E<0$, the  solutions are of the form
\begin{equation}
  \label{eq:4}
u(x)=\frac{1}{\alpha} \dn \left(\frac{x}{\beta},k\right),
\end{equation}
where $k\in\left(\frac{1}{\sqrt{2}},1\right)$ and $ b\beta^2=2\alpha^2$.
\item When $E=0$, the  solution is either $0$ or the homoclinic orbit
\(
\sqrt{\frac{-2a}{b}}\sech\left(x\sqrt{a}\right).
\)
    \item When $E>0$, the solutions are of the form
\begin{equation}
  \label{eq:5}
u(x)=\frac{1}{\alpha} \cn \left(\frac{x}{\beta},k\right),
\end{equation}
where $k\in(0,1)$ and  $b\beta^2=2k^2\alpha^2$. 
\end{itemize}

\subsection{Existence domains}
\label{sec:domains}

The various domains on which~\eqref{eq:ode} admits a solution for given parameters  $(J,E)$ are represented in Figure~\ref{fig:E-as-function-of-J}. %
From left to right we chose $(b=-1, a=1)$, $(b=1, a=1)$ and $(b=1, a=-1)$.

\begin{figure}[htpb!]
\centering
\includegraphics[width=.33\textwidth]{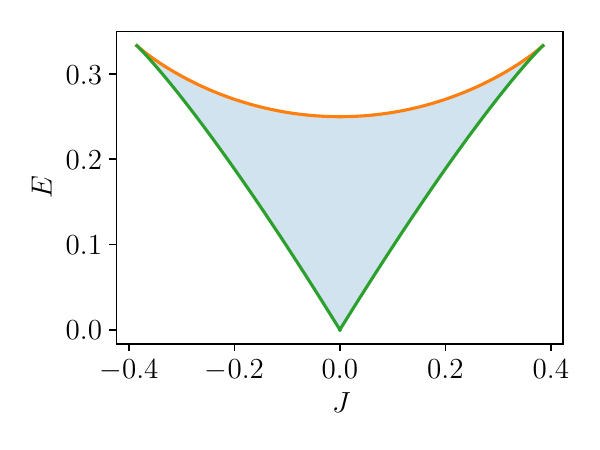}~
\includegraphics[width=.33\textwidth]{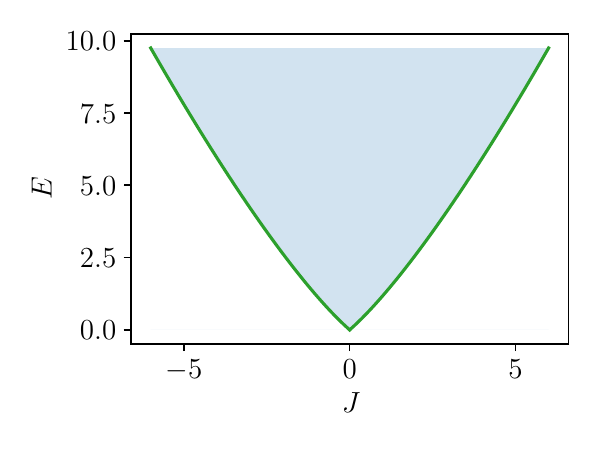}~
\includegraphics[width=.33\textwidth]{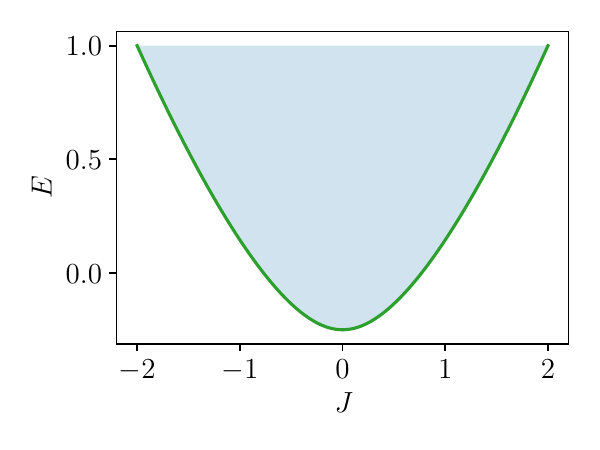}
\caption{Admissible domains for $(J,E)$. From left to right $(b,a)=(-1, 1)$, $(1, 1)$ and $(1, -1)$. The green lines are the graphs of $E_-$ and the orange line the graph of $E_+$.}
\label{fig:E-as-function-of-J}
\end{figure}

Given $u$ solution of~\eqref{eq:ode}, the function given by $\bar{u}(x)$ is also solution  of~\eqref{eq:ode}. Therefore, without loss of generality, we will always consider that
\[
  J\geq 0.
  \]
  We define the domain of existence of nontrivial quasi-periodic solutions of~\eqref{eq:ode} for each case as the following:
In the defocusing case $b<0$, let
\begin{equation} \label{eq:domain-D-defocusing-case}
D_1= \left\{ (J,E) \in \R ^2: 0<J < \sqrt{\frac{4}{27}\frac{a^3}{b^2}},\  E_{-} (J) < E < E_{+}(J) \right\}.
\end{equation}
In the focusing case $b>0$ with $a\geq0$, let 
\begin{equation} \label{eq:domain-D-focusing-case-a=1}
D_2=\left\{ (J,E) \in \R^2:J>0, E> E_{-}(J)  \right\}.
\end{equation}
In the focusing case $b>0$ with $a<0$, let 
\begin{equation} \label{eq:domain-D-focusing-case-a=-1}
D_3=\left\{ (J,E) \in \R^2: J>0, E> E_{-}(J)  \right\}.
\end{equation}
Observe that the boundaries are not included in the definition of the domains $D_1$, $D_2$, $D_3$, even though (as we described earlier) bounded solutions exist up to the boundary. The definitions of $D_2$ and $D_3$ are similar, the difference between the two stems from the function $E_{-}$, which has a different graph depending on the sign of $a$. 

\subsection{The period function}
\label{sec:period}

We describe in the following lemma the period function associated to a solution of~\eqref{eq:ode}.

\begin{lemma} \label{lem:period}
  Let $b<0$, $a>0$ ($b>0$, $a\geq0$, or $b>0$, $a<0$ respectively). Assume that $(J,E) \in D_1$ ($D_2, D_3$ respectively). Let $u$ be a solution of~\eqref{eq:ode} corresponding to $(J,E)$ and denote by $T(J,E)$ the fundamental period of $u$, which is defined in~\eqref{eq:period-T}. We denote by  $ 0 < y_1 < y_2 < y_3$, ($y_3<0 < y_1 <y_2$, $y_3 < 0 < y_1 <y_2$ respectively), the roots of the cubic polynomial $ \Pi(y)= -b y^3 -2 a y^2 + 4Ey -2J^2$. Then
\begin{equation} \label{eq:period-change-of-variable}
T(J,E)= \sqrt{2} \int_{y_1}^{y_2} \frac{dy}{\sqrt{-b(y-y_1)(y-y_2)(y-y_3)}}= 2 \sqrt{2} \int_{0}^{\frac{\pi}{2}} \frac{d\phi}{\sqrt{b(S(\phi)-y_3)}},
\end{equation}
where $ S(\phi)= y_1 \cos^2 \phi + y_2 \sin^2 \phi$.
\end{lemma}
\begin{proof}
  Notice that $\Pi(r^2)=4 r^2(E-V_J(r))$, therefore $\Pi(y)$ has three roots whenever $(J,E)\in D_1 $ ($D_2, D_3$ respectively), with $y_{1,2}=r_{1,2}^2$ for $r_{1,2}$ defined
  in Section~\ref{sec:analysis-of-the-ode}. Hence, using the change of variables $r=\sqrt{y} $ in~\eqref{eq:period-T}, we obtain the first expression in~\eqref{eq:period-change-of-variable}. 
By setting $y=S(\phi)$, we have $dy=2 \sqrt{(y-y_1)(y_2-y)} d\phi$, thus we obtain the last expression  in~\eqref{eq:period-T}.
\end{proof}

\subsection{Mass and momentum inside the domain and on the boundaries} \label{sec:Mass-Momentum}
We start by finding the mass and momentum of the solutions inside the domains $D_1, D_2$ and $D_3$, where we know that $ u(x)=r(x) e^{i \phi (x)}$. We have
\begin{align}
M(u)&= \frac{1}{2}\int_0^{T(E,J)} |u|^2  dx = \frac{1}{2} \int_0^{T(E,J)} (r(x))^2 dx \notag
\\
&= \int_{r_1}^{r_2} \frac{r^2 dr}{\sqrt{2 (E-V_J(r))}} = \sqrt{2} \int_0^{\frac{\pi}{2}} \frac{S(\phi)}{\sqrt{b(S(\phi)-y_3)}} d\phi= \tilde M(J,E),\label{tildeMPdef}
\\
P(u) &= \frac{1}{2} \Im \int_0^{T(E,J)} u \overline{u}_x dx=  \frac{1}{2} \int_0^T J dx = \frac{1}{2}TJ= \tilde P(J,E).\notag
\end{align}
The notation $\tilde{M}$ and $\tilde{P}$ is used to denote the mass and momentum as functions of $J$ and $E$.

Next we want to find the mass and momentum of the solutions on the boundaries $\{(J,E_{-}(J))\}$ for the focusing/defocusing cases and on $\{(J,E_{+}(J))\}$ for the defocusing case.

\begin{proposition} \label{prop:period-T-on-E-}
  Let $b<0$, $a>0$ ($b>0$, $a\geq0$, or $b>0$, $a<0$ respectively). Let $(J,E) \in D_1$ ($D_2, D_3$ respectively). As described in Section~\ref{sec:analysis-of-the-ode}, there exists $Q\in\left(\sqrt{\frac a3},\sqrt{a}\right)$ ($Q\in \left(\sqrt{a},\infty\right)$, $Q\in \left(0,\infty\right)$ respectively)  such that $J$ can be parametrized as $J=Q(Q^2-a)/b$. When $E \rightarrow E_{-}(J)$ the following holds:
    \begin{align*}
\lim\limits_{E \rightarrow E_{-}(J)} T(J,E)&= \frac{\pi \sqrt{2}}{\sqrt{3Q^2-a}}, \\ 
\lim\limits_{E \rightarrow E_{-}(J)} \tilde{M}(J,E)&=\frac{(Q^2-a)}{2b} \frac{\pi \sqrt{2}}{\sqrt{3Q^2-a}},
\\
\lim\limits_{E \rightarrow E_{-}(J)} \tilde{P}(J,E)&= Q \frac{(Q^2-a)}{2b} \frac{\pi \sqrt{2}}{\sqrt{3Q^2-a}}.
    \end{align*}
\end{proposition}

This proposition allows us to define the curves representing the values of $(\tilde M, \tilde P)$ along the boundary of $D_j$, $j=1,2,3$ as curves parametrized by $Q$ by
\[
  M_{\partial}(Q)=\frac{(Q^2-a)}{2b} \frac{\pi \sqrt{2}}{\sqrt{3Q^2-a}},\quad
  P_{\partial}(Q)= Q \frac{(Q^2-a)}{2b} \frac{\pi \sqrt{2}}{\sqrt{3Q^2-a}}.  
\]
Observe that $  P_{\partial}(Q)= Q M_{\partial}(Q)$. 
The curves in the different configurations are represented in Figure~\ref{fig:P-as-function-of-M}. 
\begin{figure}[htpb!]
\centering
\includegraphics[width=.3\textwidth]{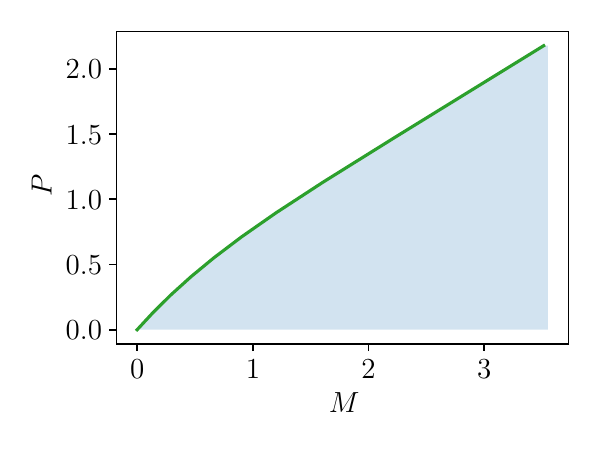}~
\includegraphics[width=.3\textwidth]{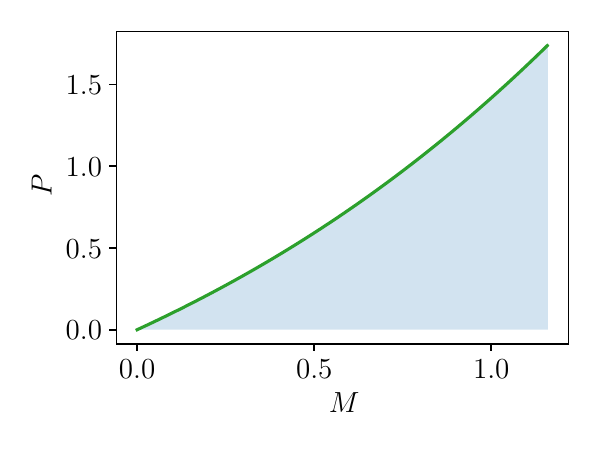}~
\includegraphics[width=.3\textwidth]{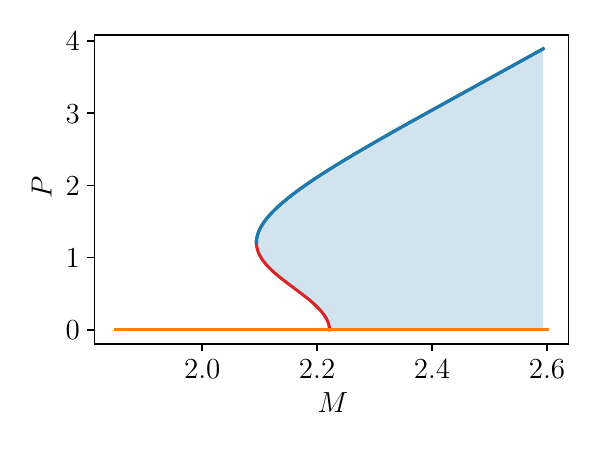}
\caption{The curve $Q\mapsto  ( M_{\partial}(Q), P_{\partial}(Q))$ for, at the left 
   $\left(b=-1, a=1,Q=\sqrt{\frac a3},\dots,\sqrt{a}\right)$, at the middle
    $\left(b=1, a=1,Q=\sqrt{a},\dots,1.5\right)$, and at the right
     $\left(b=1, a=-1,Q=0,\dots,1.5\right)$. On the right the range $Q\in\left(0, \sqrt{\frac a3}\right)$ is in red,  the range $Q\in\left(\sqrt{\frac a3},1.5\right)$ is in blue, and we have added in orange the curve $(\tilde M(0,E), \tilde P(0,E))$ for $J=0$ and $E=-\frac{a^2}{4b},\dots,12$.}

\label{fig:P-as-function-of-M}
\end{figure}

\begin{remark}
  \label{rmk:boundaries}
  We observe on Figure~\ref{fig:P-as-function-of-M} that the curves corresponding to the  boundaries of the images by $\Psi$ of the domains $D_1,D_2,D_3$  are different
 depending on $a$ and $b$. In particular, the image of the  boundary of the domain $D_3$ cannot be expected  to be the boundary of the image of $D_3$ by the mapping $(J,E)\mapsto (\tilde M, \tilde P)$, as it does not encompass the half line which is the image of $J=0$, $-\frac {a^2}{4b}< E<\infty$. Hence part of the image of the mapping is outside the shaded region in the right picture of Figure~\ref{fig:P-as-function-of-M}. 
    \end{remark}

\begin{proof}[Proof of Proposition~\ref{prop:period-T-on-E-}]
  We start from the expression~\eqref{eq:period-change-of-variable} for $T(J,E)$.
  Since $ y_1, y_2, y_3$ are solutions of the cubic equation  $-b y^3 -2 a y^2 + 4Ey -2J^2=0$, we have $ y_1+y_2+y_3=-\frac{2a}{b}$, which implies that $y_3=-\frac{2a}{b}-(y_1+y_2)$. Moreover, when $E \rightarrow E_{-}(J)$, we have $y_1,y_2\rightarrow r_Q^2=\frac{Q^2-a}{b}$. Therefore, as $E \rightarrow E_{-}(J)$, we have
  \begin{equation*}
\sqrt{b(S(\phi)-y_3)}\to \sqrt{\left( 3Q^2-a\right)}.
\end{equation*}
The limit being independent of $\phi$, we pass directly in the limit in~\eqref{eq:period-change-of-variable} to obtain

\begin{equation*}
\lim_{E \to E_{-}(J)}T(J,E) =  \frac{\pi \sqrt{2}}{\sqrt{3Q^2-a}}.
\end{equation*}

We now find the limit of the mass:
\begin{multline} \label{eq:mass-on-the-boundary}
    \lim\limits_{E \rightarrow E_{-}(J)} \tilde{M}(J,E)= \lim\limits_{E \rightarrow E_{-}(J)}\sqrt{2} \int_0^{\frac\pi2} \frac{S(\phi)}{\sqrt{b(S(\phi)-y_3)}} d\phi\\=\sqrt{2} \int_0^{\frac\pi2} \frac{\frac{Q^2-a}{b}}{\sqrt{3Q^2-a}} d\phi
    = \frac{(Q^2-a)}{2b} \frac{\pi \sqrt{2}}{\sqrt{3Q^2-a}}.
\end{multline}

The limit of the momentum is obtained by a direct computation:
\begin{equation}\label{eq:momentum-on-the-boundary}
\lim\limits_{E \rightarrow E_{-}(J)} \tilde{P}(J,E)=\lim\limits_{E \rightarrow E_{-}(J)}\frac{1}{2}T(J,E)J= \frac{1}{2}Q \frac{(Q^2-a)}{b} \frac{\pi \sqrt{2}}{\sqrt{3Q^2-a}}.
\end{equation}
This finishes the proof. 
  \end{proof}

\begin{proposition}  \label{prop:period-T-on-E+}
  Assume that $b<0$ and $a>0$.
  Let $(J,E)\in D_1$.
 When $E \rightarrow E_{+}(J)$ the following holds:
\begin{align*}
    \lim\limits_{E \rightarrow E_{+}(J)} T(J,E)&= + \infty, \\
    \lim\limits_{E \rightarrow E_{+}(J)} {\tilde{M}(J,E)}&= +\infty,&
  \lim\limits_{E \rightarrow E_{+}(J)} \frac{\tilde{M}(J,E)}{T(J,E)}&= \frac{q^2-a}{2b},\\
    \lim\limits_{E \rightarrow E_{+}(J)} {\tilde{P}(J,E)}&= +\infty,
    & \lim\limits_{E \rightarrow E_{+}(J)} \frac{\tilde{P}(J,E)}{T(J,E)}&= \frac{q(q^2-a)}{2b}.
\end{align*}
\end{proposition}

\begin{proof}
Recall that $T(J,E)= 2 \sqrt{2} \int_{0}^{\frac{\pi}{2}} \frac{d\phi}{\sqrt{b(S(\phi)-y_3)}}$ given by \eqref{eq:period-change-of-variable}, and $\tilde{M}$ and $\tilde{P}$ are defined by \eqref{tildeMPdef}. When $E \rightarrow E_{+}(J)$, both $y_2,y_3 \rightarrow r_q^2$, therefore 
\begin{equation*}
\lim\limits_{E \rightarrow E_{+}(J)} T(J,E)= \frac{2 \sqrt{2}}{\sqrt{b(y_1-y_2)}} \int_{0}^{\frac{\pi}{2}} \frac{d\phi}{\cos (\phi)}=+\infty.
\end{equation*}
Because both integrals diverge at $\phi = \pi/2$, the limit of $\frac{\tilde{M}(J,E)}{T(J,E)}$ is $\frac12S(\pi/2)|_{y_2=r_q^2}$.
The last limit is because $\frac{\tilde{P}(J,E)}{T(J,E)}=J/2$.
    \end{proof}

\subsection{Diffeomorphism} \label{sec:diffeo}

In this section we establish a diffeomorphic correspondence between $J,E$ and $M,P$. In light  of Figure~\ref{fig:P-as-function-of-M} and Remark~\ref{rmk:boundaries}, the map $(J,E)\mapsto (\tilde M,\tilde P)$ is not a diffeomorphism in the focusing case, at least when $a=-1$. Therefore, we restrict our study to the defocusing case. For simplicity in notation, in this section, we set
\[
  b=-1,\quad a=1.
\]
The arguments used still apply for generic $b<0$, $a>0$. 
Our  strategy will be based in part on the results and methods of \cite{GaHa07-2}.
We start with monotonicity results on $T$. 
\begin{proposition}
  \label{prop:T_monotone}
We have  $ \frac{\partial T}{\partial E} >0$ and $ \frac{\partial T}{\partial J} <0$ for all $ (J,E) \in D_1$. 
\end{proposition}
\begin{proof}
Let $ (J,E) \in D_1$. Since $ y_1, y_2, y_3$ are solutions of the cubic equation $ \Pi(y)=y^3 - 2y^2 + 4Ey - 2J^2 =0$, we have $ y_1+y_2+y_3=2$, and 
\begin{equation} \label{eq:partialyi-E}
\frac{\partial y_i}{\partial E}= - \frac{4y_i}{3y_i^2 - 4y_i +4E}, \quad \frac{\partial y_i}{\partial J}=  \frac{4J}{3y_i^2 - 4y_i +4E}, \quad i=1,2,3.
\end{equation}
In particular
\begin{equation*}
\frac{\partial y_1}{\partial E} <0, \quad \frac{\partial y_2}{\partial E} >0, \quad \frac{\partial y_1}{\partial E}+ \frac{\partial y_2}{\partial E}=- \frac{\partial y_3}{\partial E} >0,
\end{equation*}
because $ \Pi'(y_1) >0,\ \Pi'(y_2)<0,\ \Pi'(y_3)>0$, where $\Pi(y)= y^3 - 2y^2 + 4Ey - 2J^2$ was defined in Lemma~\ref{lem:period}. Similarly, since $J>0$,
\begin{equation*}
\frac{\partial y_1}{\partial J} >0, \quad \frac{\partial y_2}{\partial J} <0, \quad \frac{\partial y_1}{\partial J}+ \frac{\partial y_2}{\partial J}=- \frac{\partial y_3}{\partial J} <0.
\end{equation*}
On the other hand, differentiating~\eqref{eq:period-change-of-variable} with respect to $E$ and $J$, we find by chain rule, treating $y_1,y_2$ as independent variables with $E,J$ solvable from
\begin{equation}\label{EJy12}
E=y_1y_2+(y_1+y_2)y_3,\quad -2J^2=-y_1y_2y_3, \quad y_3=2-y_1-y_2,
\end{equation}
that
\begin{equation*}
\frac{\partial T}{\partial E}= A_1 \frac{\partial y_1}{\partial E}+ A_2 \frac{\partial y_2}{\partial E}, \quad \quad \quad \frac{\partial T}{\partial J}= A_1 \frac{\partial y_1}{\partial J}+ A_2 \frac{\partial y_2}{\partial J},
\end{equation*}
where $A_j = \frac{\partial T}{\partial y_j}$, $j=1,2$,
\begin{equation*}
A_1= \sqrt{2} \int_0^{\frac{\pi}{2}} \frac{1 + \cos^2 \phi }{(y_3-S(\phi))^{\frac{3}{2}}} d\phi, \quad A_2= \sqrt{2} \int_0^{\frac{\pi}{2}} \frac{1 + \sin^2 \phi }{(y_3-S(\phi))^{\frac{3}{2}}} d\phi.
\end{equation*}

We claim that
\begin{equation}\label{A2-A1}
A_2 - A_1= \sqrt{2} \int_0^{\frac{\pi}{2}} \frac{ \sin^2 \phi - \cos^2 \phi  }{(y_3-S(\phi))^{\frac{3}{2}}}  d\phi>0.
\end{equation}
Indeed, if we denote the integrand $f(\phi)= \frac{ \sin^2 \phi - \cos^2 \phi  }{(y_3-S(\phi))^{\frac{3}{2}}}  $, then for $0<\phi<\frac \pi4$ we have $S(\frac \pi 2-\phi) >S(\phi)$ due to $y_1<y_2$, and hence
$
f(\phi)+f(\frac \pi 2-\phi) >0$. Thus $A_2-A_1=2^{-1/2} \int_0^{\pi/2} [f(\phi)+f(\frac \pi 2-\phi)] d\phi>0$.

We conclude that 
\begin{equation*}
\frac{\partial T}{\partial E}= (A_2 - A_1) \frac{\partial y_2}{\partial E} + A_1 \left( \frac{\partial y_1}{\partial E} + \frac{\partial y_2}{\partial E} \right) >0.
\end{equation*}
Similarly, 
\begin{equation*}
\frac{\partial T}{\partial J}= (A_2 - A_1) \frac{\partial y_2}{\partial J} + A_1 \left( \frac{\partial y_1}{\partial J} + \frac{\partial y_2}{\partial J} \right) <0.
\end{equation*}
This concludes the proof.
\end{proof}
For all $(J,E) \in D_1$, let%
\begin{equation*}
\Delta = \begin{vmatrix}
\frac{\partial \tilde P}{ \partial E} & \frac{\partial \tilde M}{\partial E} \\[4pt]
\frac{\partial \tilde P}{ \partial J} & \frac{\partial \tilde M}{\partial J}
\end{vmatrix}.
\end{equation*} 
\begin{proposition}
For all $(J,E) \in D_1$ we have $\Delta <0$.
\end{proposition}
\begin{proof}
Recall that $A_1$ and $A_2$ are defined in the proof of Proposition~\ref{prop:T_monotone}. We have
\begin{align*}
\frac{\partial \tilde P}{\partial E} &= \frac{1}{2} J \frac{\partial T}{\partial E}=  \frac{1}{2} J \left( A_1 \frac{\partial y_1}{\partial E} + A_2 \frac{\partial y_2}{\partial E} \right),\\
\frac{\partial \tilde P}{\partial J} &= \frac{1}{2} \left( J \frac{\partial T}{\partial J}+ T \right)= \frac{1}{2} \left( J \left( A_1 \frac{\partial y_1}{\partial J}+ A_2 \frac{\partial y_2}{\partial J} \right) +T \right),\\
\frac{\partial \tilde M}{\partial E} &= C_1 \frac{\partial y_1}{\partial E} + C_2 \frac{\partial y_2}{\partial E}, \qquad
\frac{\partial \tilde M}{\partial J} = C_1 \frac{\partial y_1}{\partial J} + C_2 \frac{\partial y_2}{\partial J} .
\end{align*}
In the last two derivatives we have used chain rule treating $y_1,y_2$ as independent variables with $E,J$ solvable from \eqref{EJy12}, and
\begin{align*}
C_1&=\frac{\partial \tilde M}{\partial y_1} =
 \sqrt{2} \int_0^{\frac{\pi}{2} } 
\left\{ 
 \frac{ \cos^2 \phi}{ \sqrt{y_3-S(\phi)}} +   \frac{(1+\cos^2 \phi) S(\phi) }{2 \sqrt{y_3- S(\phi)}^3} \right\} d\phi ,\\
C_2&=\frac{\partial \tilde M}{\partial y_2} =
 \sqrt{2} \int_0^{\frac{\pi}{2} } 
\left\{ 
 \frac{ \sin^2 \phi}{ \sqrt{y_3-S(\phi)}} +   \frac{(1+\sin^2 \phi) S(\phi) }{2 \sqrt{y_3- S(\phi)}^3} \right\} d\phi .
\end{align*}

Therefore,
\begin{align*}
2\Delta &= \left( J \left( A_1 \frac{\partial y_1}{\partial E} + A_2 \frac{\partial y_2}{\partial E} \right)\right) \left(C_1 \frac{\partial y_1}{\partial J} + C_2 \frac{\partial y_2}{\partial J} \right) \\
&\qquad -   \left( J \left( A_1 \frac{\partial y_1}{\partial J}+ A_2 \frac{\partial y_2}{\partial J} \right) +T  \right) \left(C_1 \frac{\partial y_1}{\partial E} + C_2 \frac{\partial y_2}{\partial E} \right) \\
&= \left( A_1 C_2 -A_2 C_1 \right)  J  \left(\frac{\partial y_1}{\partial E} \frac{\partial y_2}{\partial J}- \frac{\partial y_2}{\partial E} \frac{\partial y_1}{\partial J}  \right) -  T \frac{\partial \tilde M}{\partial E}.
\end{align*}
Using the identity~\eqref{eq:partialyi-E}, we have
\begin{equation*}
\frac{\partial y_i}{\partial E} = -\frac{y_i}{J} \frac{\partial y_i}{\partial J}.
\end{equation*} 
Therefore
\begin{equation}\label{Wronskian}
\Delta= \frac{1}{2} (  A_1 C_2-A_2 C_1 ) (y_2 -y_1) \left( \frac{\partial y_1}{\partial J} \frac{\partial y_2}{\partial J} \right) - \frac{1}{2} T \frac{\partial \tilde M}{\partial E}. 
\end{equation}

We now rewrite $C_1$ and $C_2$,
\begin{align*}
C_1&=
 \sqrt{2} \int_0^{\frac{\pi}{2} } 
\left\{ 
 \frac{ \cos^2 \phi}{ \sqrt{y_3-S(\phi)}} +   \frac{(1+\cos^2 \phi) [y_3-(y_3-S(\phi)] }{2 \sqrt{y_3- S(\phi)}^3} \right\} d\phi= \frac {y_3}2 A_1 - \tilde C_1, \\
\tilde C_1&=\frac {\sqrt2}2\int_0^{\frac{\pi}{2}} \frac{\sin^2 \phi  }{(y_3-S(\phi))^{\frac{1}{2}}} d\phi.
\end{align*}
Similarly,
\[
C_2=  \frac {y_3}2 A_2 - \tilde C_2,\quad 
\tilde C_2=\frac {\sqrt2}2\int_0^{\frac{\pi}{2}} \frac{\cos^2 \phi  }{(y_3-S(\phi))^{\frac{1}{2}}} d\phi.
\]
By the same argument we showed $A_2-A_1>0$ below \eqref{A2-A1}, we have $\tilde C_1-\tilde C_2 >0$. Hence
\[
C_2-C_1 =\frac {y_3}2 (A_2-A_1) +\tilde C_1-\tilde C_2>0.
\]
As a consequence, by the inequalities below \eqref{eq:partialyi-E} we have
\begin{equation}\label{pMpE}
\frac{\partial \tilde M}{\partial E} = C_1 \frac{\partial y_1}{\partial E} + C_2 \frac{\partial y_2}{\partial E} = ( C_2 - C_1) \frac{\partial y_2}{\partial E} + C_1 ( \frac{\partial y_1 }{\partial E} + \frac{\partial y_2}{\partial E})>0.
\end{equation}
Moreover,
\begin{align*}
A_1 C_2 - A_2 C_1&= A_1  ( \frac {y_3}2 A_2 - \tilde C_2) - A_2 ( \frac {y_3}2 A_1 - \tilde C_1)\\
&= 
A_2  \tilde C_1 - A_1  \tilde C_2
>0.
\end{align*}
The above show that both terms on the right side of \eqref{Wronskian} are negative. 
Hence $\Delta <0$,
which concludes the proof.
\end{proof}

Let
\[
\tilde{D}=\left\{ (\tilde M,\tilde P) \in \R^2: \tilde M>0,\ 0<\tilde P < \frac {\tilde M}{\pi} {\sqrt{3 \tilde M^2+ \pi^2-\sqrt{9 \tilde M^4+4 \tilde M^2\pi^2}}} \right\}.
\]
The domain $\tilde{D}$ is the domain represented on the left of Figure~\ref{fig:P-as-function-of-M}.

\begin{theorem} \label{the:diffeo}
The map $\Psi : D_1 \to \tilde{D}$, defined by $\Psi(J,E)= (\tilde M(J,E),\tilde P(J,E))$
is a diffeomorphism.
\end{theorem}
\begin{proof}
  We know from Section~\ref{sec:Mass-Momentum} that the following hold:
\begin{gather*}
  \lim\limits_{E \rightarrow E_{-}(J)} T(J,E)= \frac{\pi \sqrt{2}}{\sqrt{3Q^2-1}},\quad \lim\limits_{E \rightarrow E_{+}(J)} T(J,E)= + \infty,\\
  \lim\limits_{E \rightarrow E_{-}(J)} \tilde M(J,E)= \frac{(1-Q^2)}{2}\frac{\pi \sqrt{2}}{\sqrt{3Q^2-1}}, \quad
\lim\limits_{E \rightarrow E_{-}(J)} \tilde P(J,E)= \frac{1}{2}Q (1-Q^2) \frac{\pi \sqrt{2}}{\sqrt{3Q^2-1}},\\
  \lim\limits_{E \rightarrow E_{+}(J)} \tilde M(J,E) = + \infty, \quad  \lim\limits_{E \rightarrow 0} \tilde M(0,E) =0,\\
\quad  \lim\limits_{E \rightarrow E_{+}(J)} \tilde P(J,E)= \infty.
\end{gather*}
The range of the map $ \tilde M:D_1 \rightarrow \R $ is exactly the interval $(0,+\infty)$ since $\frac{\partial \tilde M}{\partial E} >0 $ by \eqref{pMpE}. We fix $ \tilde M_0>0$, and let $\Sigma= \{ (J,E) \in D_1 : \tilde M(J,E)= \tilde M_0\}$.

By the implicit Function Theorem, $ \Sigma$ is a smooth curve in $D_1$ which can be represented as a graph over the $J$-axis. Let $E_0$ be the energy such that $\tilde M(0,E_0)= \tilde M_0$. Let 
$(J_0,E_{-}(J_0))$ be determined by the relation
\begin{equation*}
\frac{(1-Q_0^2)}{2}\frac{\pi \sqrt{2}}{\sqrt{3Q_0^2-1}} = \tilde M_0,
\end{equation*}
where $J_0=Q_0(1-Q_0^2)$, $Q_0 \in ( \frac{1}{\sqrt{3}},1) $. The curve $ \Sigma$ connects the boundary points $ (0,E_0)$ and $(J_0,E_{-}(J_0))$.
We know that $\Delta <0$ which implies that the restriction of $\tilde P$ to the curve $\Sigma$ is a strictly increasing function of $J$, because
\begin{equation*}
\frac{d}{d J}\tilde P |_{\Sigma}= \left( \frac{\partial \tilde M}{\partial E} \right) ^{-1} \left( \frac{\partial \tilde M}{\partial E} \frac{\partial \tilde P}{\partial J} - \frac{\partial \tilde M}{\partial J} \frac{\partial \tilde P}{\partial E} \right) >0 .
\end{equation*}
Thus $\tilde P$ varies from $0$ to $ \tilde P_0$ on the curve $\Sigma$, where
\begin{equation*}
\tilde P_0= \frac{1}{2}Q_0 (1-Q_0^2) \frac{\pi \sqrt{2}}{\sqrt{3Q_0^2-1}}=\frac{\tilde M_0}{\pi}{\sqrt{3 \tilde M_0^2+ \pi^2-\sqrt{9 \tilde M_0^4+4 \tilde M_0^2\pi^2}}}.
\end{equation*}
This proves that $\Psi$ is onto $\tilde{D}$ and therefore $\Psi$ is a diffeomorphism.
\end{proof}

\section{From the minimization problem to the ordinary differential equation} \label{sec:min-ode}
In this section, we aim to study the minimization problem~\eqref{eq:minim-problem} and establish the link between its minimizers and the solutions of the ordinary differential equation~\eqref{eq:ode}. One of the difficulties is identifying to which solution of~\eqref{eq:ode} a minimizer of~\eqref{eq:minim-problem} will correspond. 
The existence of minimizers can be proved by taking a sequence that minimizes the energy and showing that it is bounded in $H^{\theta}_T$, which implies that it has a weakly convergent subsequence in $H^{\theta}_T$. The weak convergence of the subsequence, along with the compactness of the embedding of $H^{\theta}_T$ into $L^p$ for $p \in (1, \infty)$, implies the convergence of the subsequence in $L^p$. Therefore, the minimizing sequence converges to a function that minimizes the energy.
Solutions of the minimization problem~\eqref{eq:minim-problem} were studied in particular cases in \cite{GuLeTs17}, in particular when the Floquet multiplier is either $\theta=0$ or $\theta=\pi$, and when the momentum constraint is either $0$ or absent. The general case was left open. In the present section, we will show that when the parameters of~\eqref{eq:minim-problem} are chosen to match those of a solution of~\eqref{eq:ode}, then this solution is obtained \emph{numerically} as the minimizer.

\subsection{Variational problems} \label{sec:variational-problems}

We start by recalling the result of Gustafson, Le Coz and Tsai \cite{GuLeTs17} on the global variational characterizations of the elliptic function periodic waves as constrained-mass energy minimizers among periodic functions. Later on, we use these results to compare our numerical method with the theoretical results.

For $m>0$, we consider the following minimization problem with fixed mass
\begin{equation} \label{eq:minimization-prob}
\min \{ \mathcal{E}(u): u \in H^{\theta}_T,\, M(u)=m\},
\end{equation}
and the minimization problem with fixed mass and momentum
\begin{equation} \label{eq:minimization-probl}
\min \{ \mathcal{E}(u): u \in H^{\theta}_T,\, M(u)=m,\, P(u)=0\}.
\end{equation}

Assume $b>0$ (focusing case). We have the following properties (see \cite[Propositions 3.2, 3.4]{GuLeTs17}).
\begin{enumerate}
\item Let $\theta=0$, we are in the periodic case:
\begin{enumerate}
\item If $0<m\leq \frac{\pi^2}{bT}$ then there exists a unique (up to phase shift) minimizer for~\eqref{eq:minimization-prob} and~\eqref{eq:minimization-probl}, which is the constant function $u_{min}=\sqrt{\frac{2m}{T}}$.
\item If $\frac{\pi^2}{b T}<m< \infty $ then there exists a unique minimizer (up to translations and phase shift) for~\eqref{eq:minimization-prob} and~\eqref{eq:minimization-probl}, which is the rescaled function $\dn_{\alpha, \beta,k}=\frac{1}{\alpha}\dn \left(\frac{.}{\beta},k \right)$, where the parameters $\alpha, \beta,k$ are uniquely determined. In particular given $k\in (0,1)$, if $b=2$, $T=2 \mathbf{K}(k)$ and $m= \mathbf{E}(k)$, then the unique minimizer is $\dn(.,k)$.
\end{enumerate} 
\item If $\theta=\pi$, we are in the anti-periodic case, then there exists a unique minimizer (up to translations and phase shift) for~\eqref{eq:minimization-prob} and~\eqref{eq:minimization-probl}, which is the rescaled function $\cn_{\alpha, \beta,k}=\frac{1}{\alpha}\cn \left(\frac{.}{\beta},k \right)$, where the parameters $\alpha, \beta,k$ are uniquely determined. In particular given $k\in (0,1)$, if $b=2 k^2$, $T=2 \mathbf{K}(k)$ and $m= \frac{\mathbf{E}-(1-k^2)\mathbf{K}}{k^2}$, then the unique minimizer is $\cn(.,k)$.
  \end{enumerate}

Assume now that $b<0$ (defocusing case). We have the following properties (\cite[Propositions 3.3, 3.6]{GuLeTs17}).
\begin{enumerate}
\item If $\theta=0$, we are in the periodic case, then there exists a unique minimizer (up to phase shift) for~\eqref{eq:minimization-prob} and~\eqref{eq:minimization-probl}, which is the constant function $u_{min}=\sqrt{\frac{2m}{T}}$.
\item If $\theta=\pi$, we are in the anti-periodic case: There exists a unique minimizer (up to phase shift) for~\eqref{eq:minimization-prob}  which is the plane wave $u_{min}=\sqrt{\frac{2m}{T}} e^{\frac{2i \pi x}{T}}$.
\end{enumerate}

    The following was conjectured, and confirmed numerically in \cite{GuLeTs17}. 
There exists a unique minimizer (up to translations and phase shift) of~\eqref{eq:minimization-probl}  which is the rescaled function $\sn_{\alpha, \beta,k}=\frac{1}{\alpha}\sn \left(\frac{.}{\beta},k \right)$, where the parameters $\alpha, \beta,k$ are uniquely determined. In particular given $k\in (0,1)$, if $b=-2 k^2$, $T=2 \mathbf{K}(k)$ and $m=\frac{(\mathbf{K}-\mathbf{E})}{k^2}$, then the unique minimizer is $\sn(.,k)$. 

\subsection{Numerical solutions of the ordinary differential equation} \label{sec:ode}
In this section we explain how we proceeded to compute numerically the solution of the ordinary differential equation~\eqref{eq:ode}:
\begin{equation*}
u_{xx}+au+b |u|^2 u=0.
\end{equation*}
The Python function \textit{odeintw} has been used to find the numerical solution. In order to use this function we need to find the period $T$ and the initial data at a fixed $(J,E)$.
Recall that $V_J(r)$ was defined in~\eqref{eq:def-E-V_J}. 
We start our numerical experiments by fixing a value of $J$ in the domain $D_1$ ($D_2$, $D_3$ respectively) for each defocusing and focusing case. Recall that when $J\neq 0$, $u$ is of the form $u(x)= r(x) e^{i \phi (x)} $.
If we are in the defocusing case we find the minimum $r_Q$ and the maximum $r_q$ of $V_J(r)$. We know that $E_{-}=V_J(r_Q)$ and $E_{+}=V_J(r_q)$. Then we choose $E$ such that $E_{-} \leq E < E_{+}$. Note that we chose $ E < E_{+}$ since $\lim\limits_{E \rightarrow E_{+}(J)} T(J,E)= + \infty$. Moreover we find the three positive roots $r_1<r_2<r_3$ of $E-V_J(r)$. 
If we are in the focusing case we find the minimum $r_Q$ of $V_J(r)$ and we have $E_{-}=V_J(r_Q)$. Then we choose $E$ such that $E_{-} \leq E $. Moreover we find the two positive roots $r_1<r_2$ of $E-V_J(r)$.
Recall that the period $T$ is given by~\eqref{eq:period-T}
\begin{equation*}
T(J,E)= 2\int_{r_1}^{r_2} \frac{1}{\sqrt{2(E-V_J(r))}}dr.
\end{equation*}
We will obtain similar expressions for Floquet multiplier, mass and momentum.
The Floquet multiplier is given by
\begin{multline*}
\theta(J,E)=\phi(x+T)-\phi(x)=\int_0^T \phi_x(x) dx =\int_0^T \frac{-J}{r(x)^2} dx = 2 \int_{r_1}^{r_2} \frac{-J}{r^2 \sqrt{2(E-V_J(r))}}dr.
\end{multline*}
The mass is given by
\begin{equation*}
M(J,E)=\frac{1}{2}\int_0^T |u|^2 dx =\frac{1}{2}\int_0^T r(x)^2 dx = \int_{r_1}^{r_2} \frac{r^2}{\sqrt{2(E-V_J(r))}}dr.
\end{equation*} 
The momentum is given by
\begin{multline*}
  P(J,E)=\frac{1}{2} \Im \int_0^T u \overline{u}_x dx=-\frac{1}{2}\int_0^T \phi_x r(x)^2 dx %
  =\frac{1}{2} \int_0^T J dx= \int_{r_1}^{r_2} \frac{J}{\sqrt{2(E-V_J(r))}}dr.
\end{multline*} 
Recall that on the boundary $E=E_{-}$ the period $T$ is given by the limit when $E$ tends to $E_{-}$, which is given in Proposition~\ref{prop:period-T-on-E-}.

To use the function \textit{odeintw}, we introduce the variable $v=u'$ and rewrite~\eqref{eq:ode} as a first order equation:
\[
\begin{pmatrix} 
u    \\ 
v
\end{pmatrix} '= \begin{pmatrix} 
v    \\ 
-au-b|u|^2 u 
\end{pmatrix}.
  \]
  The initial data will be given by $u(0)=r(0)=r_1$ and $u'(0)= i \phi'(0)r(0) =-\frac{J}{r_1} i$.

The results of the numerical experiments are presented in Section~\ref{subsec:experiments}.

\subsection{Continuous gradient flow with discrete normalization}
\label{sec:gradient-flow}
     
Our goal in this section is to present our approach for the computation of solutions of the minimization problem~\eqref{eq:minim-problem}.
We introduce an approach based on a normalized gradient flow: at each step we evolve in the direction of the gradient of the energy and renormalize the mass and the momentum of the outcome.
Compared to approach in the existing literature (see e.g.~\cite{BaDu04,FaJe18, GuLeTs17}), the key difference is that we  treat two constraints at the same time. 

Define an increasing sequence of time $0=t_0<...<t_n$ and take an initial data $u_0$. Between each time step, let $u(t,x)$ evolve along the gradient flow
\begin{equation} \label{eq:gradient-flow}
\left \{
\begin{array}{lcl}
u_t= - \mathcal{E}'(u)=u_{xx}+b|u|^2u,\\
u(t_n,x)=u_n(x),
\end{array}
\right.
\end{equation}
where $x \in \R$, $t_n<t<t_{n+1}$, $n \geq 0$.
At each time step $t_n$, the function is renormalized so as to have the desired mass $m$ and momentum $p$. Recall that, as noted in \cite{BaDu04}, the renormalization of the mass is equivalent to solving exactly the following ordinary differential equation:
\begin{equation}
u_t=\mu_nu, \quad t_n<t<t_{n+1}, \quad n\geq 0, \quad \mu_n= \frac{1}{t_{n+1}-t_n} \ln \left(\frac{\sqrt{2m}}{\|u(t_n)\|_{L^2}} \right).
\end{equation} 
Inspired by this remark, we consider the following numerical algorithm which gives the desired mass $m$ and momentum $p$ simultaneously by one flow. Suppose $ {\tilde{u}}_{n+1}$ has been computed from $u_n$ (e.g.~using a semi-implicit Euler scheme). To normalize $ {\tilde{u}}_{n+1}$ we proceed in the following way. Let
\begin{equation*}
m_0= M({\tilde{u}}_{n+1}), \quad p_0=P({\tilde{u}}_{n+1}), \quad k_0=\frac{1}{2} \int_0^T | \partial_x \tilde{u}_{n+1}|^2 dx.
\end{equation*}
Consider the following flow
\begin{equation*}
\partial_t u=( \mu + i \omega \partial_x) u.
\end{equation*}
We have at ($t=t_{n+1}$)
\begin{equation*}
\frac{d}{dt} M(u)=\Re \int _0^T \bar{u} u_t dx =2m_0\mu+2p_0\omega,
\end{equation*}
\begin{equation*}
\frac{d}{dt} P(u)= \Im \int_0 ^T \bar{u}_x u_t dx= 2p_0\mu+ 2k_0\omega.
\end{equation*}
We want to choose $\mu_n=\mu \delta t$ and $\omega_n= \omega \delta t$ so that
\begin{equation*}
2m_0\mu_n +2p_0\omega_n=m-m_0,
\end{equation*}
\begin{equation*}
2p_0 \mu_n+ 2 k_0 \omega _n= p-p_0.
\end{equation*}
Note that by Cauchy-Schwarz 
\begin{equation*}
p_0^2=\left( \frac{1}{2} \Im \int_0^T u \bar u_x dx \right)^2 \leq \left( \frac{1}{2} \int_0^T |u|^2 dx \right) \left( \frac{1}{2} \int_0^T |u_x|^2 dx \right) = m_0k_0.
\end{equation*}
In the case $m_0k_0-p_0^2>0$, we can solve 
\begin{equation*}
\mu_n=\frac{k_0(m-m_0)-p_0(p-p_0)}{2(m_0k_0-p_0^2)},
\end{equation*}
\begin{equation*}
\omega_n=\frac{m_0(p-p_0)-p_0(m-m_0)}{2(m_0k_0-p_0^2)}.
\end{equation*}
In the case $m_0k_0-p_0^2=0$, we are in the case of equality of the Cauchy-Schwarz inequality. Therefore, $\tilde{u}_{n+1}$ is a plane wave and for some real constants $r, \alpha$ we have
\begin{equation*}
\tilde{u}_{n+1}=r e^{i \alpha x}.
\end{equation*}
In that case, we choose
\begin{equation*}
u_{n+1}= \sqrt{\frac{2 m}{T}} e^{-i\frac{p}{m}x}.
\end{equation*}

\subsection{Discretization}
\label{sec:discretization}

Several numerical methods can be considered for discretizing~\eqref{eq:gradient-flow}, for example a standard Crank-Nicolson scheme. We chose a semi-implicit Euler scheme similar to the one used in \cite{BaDu04}. The main novelty of our approach is in the normalization step which completes the scheme. The (time) discretized scheme that we use is the following:
\begin{equation*}
\left \{
\begin{aligned}
\frac{{\tilde{u}}_{n+1}-u_n}{\delta_t} &= \partial_{xx} {\tilde{u}}_{n+1} + b|u_n|^2{\tilde{u}}_{n+1}, \\ 
\frac{u_{n+1}-{\tilde{u}}_{n+1}}{\delta_t}&= (\mu_n+i \omega_n \partial _x)u_{n+1},
\end{aligned}
\right.
\end{equation*}
where $\mu_n$ and $\omega_n$ are defined in Section~\ref{sec:gradient-flow}.
Note that, due to our choice of a semi-implicit scheme, the first equation is linear in ${\tilde{u}}_{n+1}$
and can be easily solved. The same holds for the second equation.   
Finally we present the fully discretized problem. We discretize the space interval $[0,T]$ by setting
\begin{equation*}
x^0=0, \quad x^l=x^0+l \delta x, \quad \delta x= \frac{T}{L}, \quad L \in \N.
\end{equation*}
We denote by $u_n^l$ the numerical approximation of $u(t_n,x^l)$. Using the (Backward Euler) semi-implicit scheme for time discretization  and second-order centered finite difference for the second spatial derivatives, and the second order centered finite difference for the first spatial derivatives, we obtain the following scheme:
\begin{equation*}
\left \{
\begin{aligned}
\frac{{\tilde{u}}_{n+1}^l-u_n^l}{\delta_t}=& \frac{{\tilde{u}}_{n+1}^{l-1}-2{\tilde{u}}_{n+1}^l+{\tilde{u}}_{n+1}^{l+1}}{\delta x^2} + b|u_n^l|^2{\tilde{u}}_{n+1}^l, \\
\frac{u_{n+1}^l-{\tilde{u}}_{n+1}^l}{\delta_t}&= \mu u_{n+1}^l+ i \omega \frac{u_{n+1}^{l-1}-u_{n+1}^{l+1}}{2 \delta x}.
\end{aligned}
\right.
\end{equation*}

Recall that as we are in the space $H^{\theta}_T$, we know that $u(x+T)=e^{i \theta} u(x)$, therefore with the discretization on the space interval $[0,T]$, we have
\begin{equation*}
u_{n+1}^L= e^{i \theta} u_{n+1}^0
\end{equation*}

\begin{remark}
Since the problem is invariant under phase multiplication and spatial translation, at each step we shift $u_j$ so that the minimum of its modulus is at the boundary, and is real. We do the same for the numerical solutions of the ordinary differential equation. 
\end{remark}

\subsection{Experiments} \label{subsec:experiments}
Our main observation is the following.
Let $(J,E) \in D_1$ ($D_2$, $D_3$ respectively) for each defocusing and focusing case. Let $u_{ode}$ be the associated solution of the ordinary differential equation~\eqref{eq:ode}. Consider the minimization problem
\begin{equation}
    \label{eq:min-prob-ode}
\min \{ \mathcal{E}(u): u \in H^{\theta}_T,\, M(u)=M(u_{ode}),\, P(u)=P(u_{ode}), u \in H_T^{\theta} \}.
\end{equation}
The minimizer of~\eqref{eq:min-prob-ode} obtained with the algorithm described in Section~\ref{sec:discretization} is given up to admissible numerical errors by $u_{ode}$.

We performed different tests using the scheme  above described. In every test we are comparing the minimizer that we obtained from the normalized gradient flow with the numerical solution of the ordinary differential equation.
The choice of the initial data was arbitrary and equal to $ u_0(x)=1+ i+ \cos \left(\frac{2 \pi x}{T} \right)$. Other initial data lead to similar results. We use $L=1000$ grid points for the interval $[0,T]$. We run the algorithm until a maximal difference of $\epsilon=10^{-6}$ between the absolute values of the moduli of $u_j^l$. For the time discretization we take $\delta_t=10^{-3}$.

Minimization among periodic functions is completely covered by theoretical results (see Section~\ref{sec:variational-problems}).
Therefore we started the experiments by verifying on these cases that our algorithm recovers the expected theoretical results. We have chosen to fix $k=0.9$. The others parameters will be chosen so as to obtain $\cn$, $\dn$ and $\sn$ (recall that to run the algorithm we have to fix $a,b,J$ and $E$). For all three cases we have $J=0$.

We start with the $\dn$ case, which corresponds to
\begin{equation*}
E=-0.095,\quad b=2, \quad a=-(2-k^2).
\end{equation*}
As mentioned in Section~\ref{sec:ode}, after fixing $J$ and $E$ we can now find the period $T$, the Floquet  multiplier $\theta$, the mass $m$ and the momentum $p$. We found the following quantities:
\begin{equation*}
T=2 \mathbf{K}(k), \quad  \theta =0, \quad m=\mathbf{E}(k), \quad p=0,
\end{equation*}
which are the period, mass and momentum of the dnoidal function $\dn$.
We start by plotting $u_{ode}$ the numerical solution of the ordinary differential equation~\eqref{eq:ode} and the function $\dn(x,k)$. We can see on the left of Figure~\ref{fig:ode-dn} that $u_{ode}$ is very close to the exact solution $\dn$. 
The second step is to run the algorithm and find the minimizer $u_{min}$ which we compare to the solution of the ODE $u_{ode}$. We observe convergence towards dnoidal functions as we see on the right of Figure~\ref{fig:ode-dn} with a maximum difference between the solutions of $3 \times 10^{-3}$.
\begin{figure}[htpb!]
\centering
\includegraphics[width=.49\textwidth]{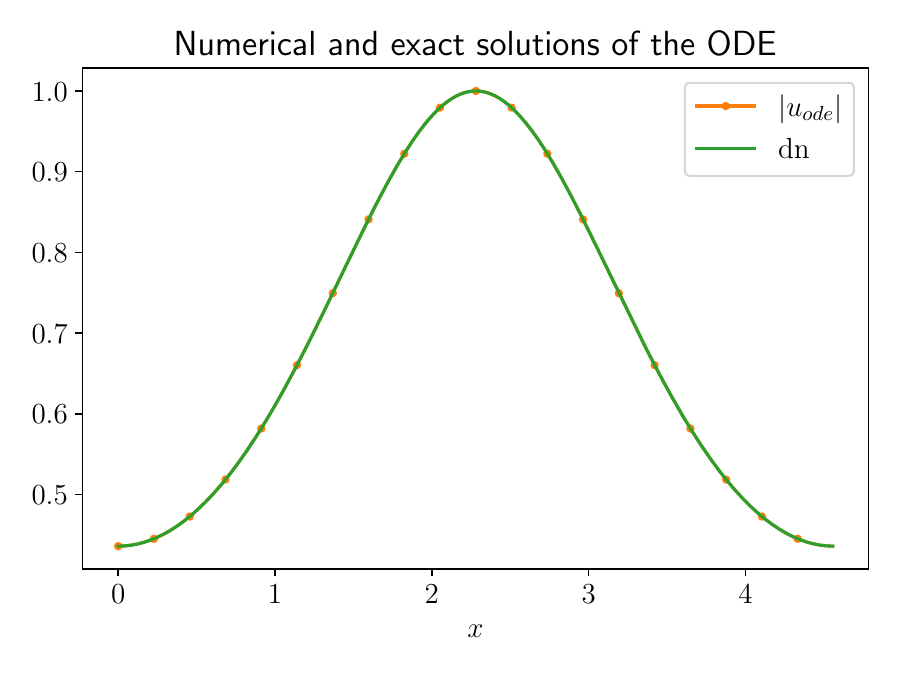}~
\includegraphics[width=.49\textwidth]{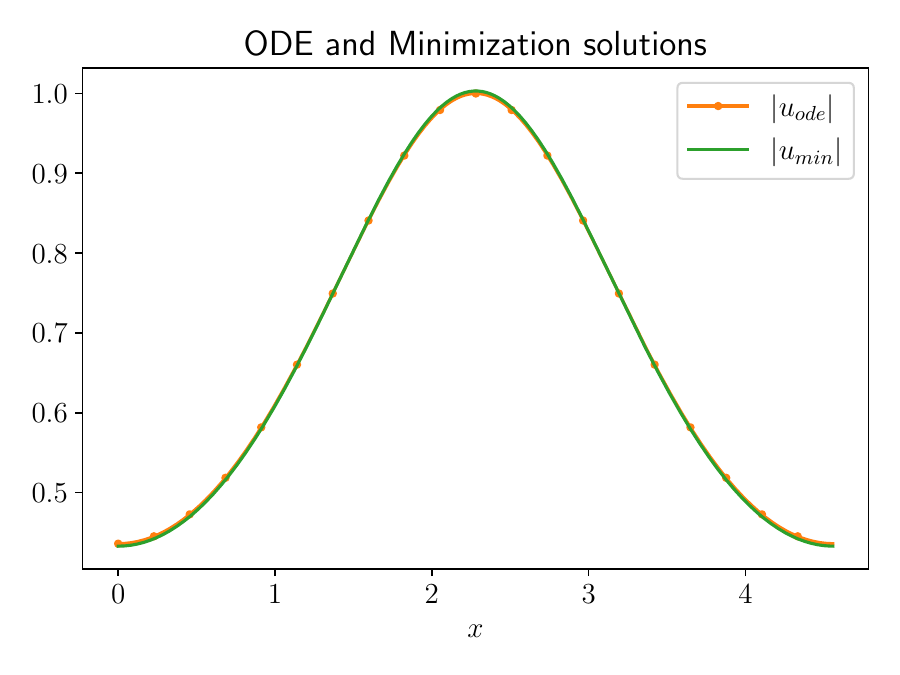}~
\caption{Comparison of the $\dn$ function and the solutions of the ODE and the minimization problem.}
\label{fig:ode-dn}
\end{figure} 

The second case is the case of $\cn$. We fix 
\begin{equation*}
E=0.095, \quad b=2k^2,\quad a=1-2k^2.
\end{equation*}
As before we find the period $T$, the  mass  $m$ and the momentum $p$, and the Floquet  multiplier $\theta$:
\begin{equation*}
T=2\mathbf{K}(k), \quad \theta=\pi, \quad m=\frac{\mathbf{E}-(1-k^2)\mathbf{K}}{k^2}, \quad  p=0.
\end{equation*}
We start by comparing the numerical solution of the ODE with the exact solution. We can see on the left of Figure~\ref{fig:ode-cn} that the numerical solution of the ODE is $\cn$ up to numerical errors.
Then we compare it to the solution  of the minimization  problem. We observe convergence towards cnoidal functions as we see on the right of Figure~\ref{fig:ode-cn} with a maximum difference between the solutions of $3 \times 10^{-4}$.
\begin{figure}[htpb!]
\centering
\includegraphics[width=.45\textwidth]{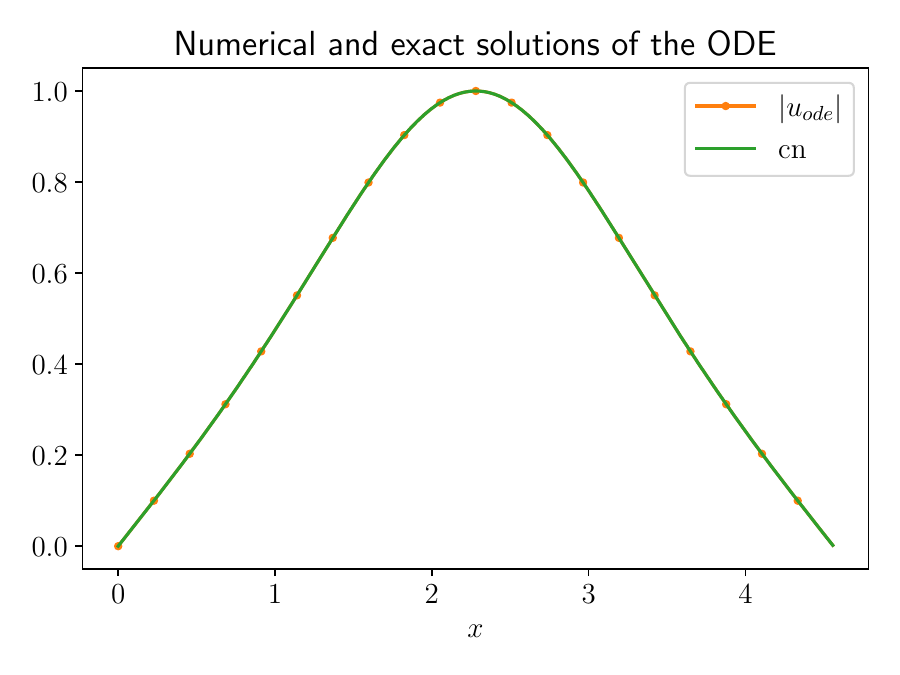}~
\includegraphics[width=.45\textwidth]{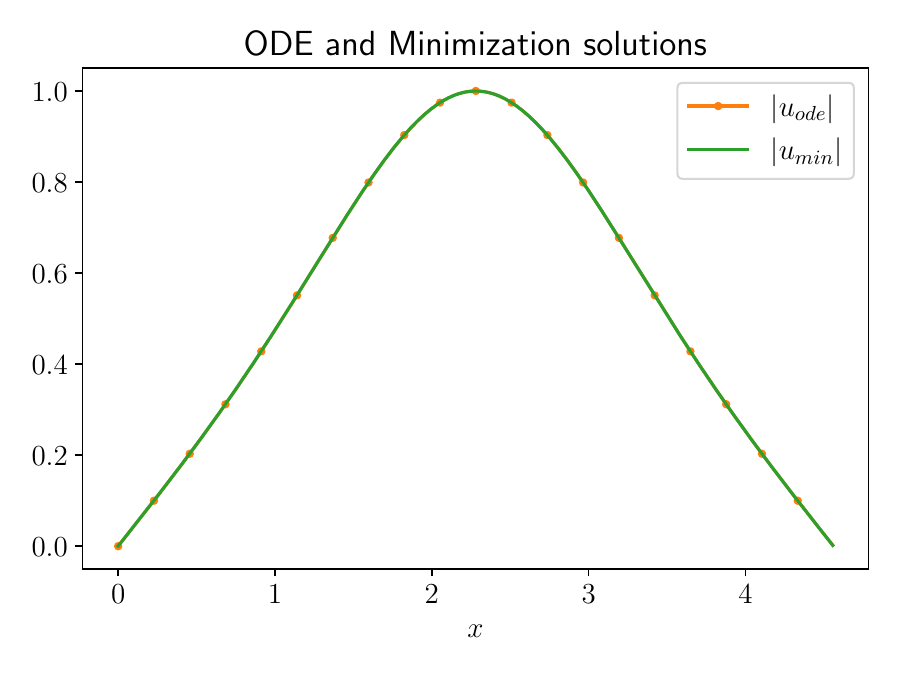}~
\caption{Comparison of the $\cn$ function and the solutions of the ODE and the minimization problem.}
\label{fig:ode-cn}
\end{figure} 

For the third case, let 
\begin{equation*}
E=0.5, \quad b=-2k^2,\quad a=1+k^2.
\end{equation*}
With the same method as before we have
\begin{equation*}
T=2\mathbf{K}(k), \quad \theta=\pi, \quad m=\frac{\mathbf{K}-\mathbf{E}}{k^2}, \quad  p=0.
\end{equation*}
First we can see on the left of Figure of~\ref{fig:ode-sn} that the numerical solution of the ODE is the exact solution $\sn$. 
Then we compare it to the solution  of the minimization  problem. We observe convergence towards snoidal functions as we  see on the right of Figure~\ref{fig:ode-sn} with a maximum difference between the solutions of $1 \times 10^{-2}$. This case is not covered by the theoretical result. Our observation here is similar to the one of \cite{GuLeTs17}: snoidal functions minimize the energy on fixed mass and $0$ momentum among anti-periodic functions.

Now we will test the conjecture on the other cases where we do not know the theoretical results, i.e.~for complex valued solutions of~\eqref{eq:ode}. To do so we fix different values of $J$ and for each of these values, we choose $E$ such that $(J,E) \in D_1$, ($(D_2),(D_3)$) defined in~\eqref{eq:domain-D-defocusing-case},~\eqref{eq:domain-D-focusing-case-a=1},~\eqref{eq:domain-D-focusing-case-a=-1}. We plot the numerical solution of the ODE and the solution of the minimization problem and we do the comparison.

We start with the defocusing case, and we fix $b=-1$ and $a=1$. We choose $J=0.2$ arbitrarily. For other values of $(J,E) \in D_1$ that we tested, we obtain the same result. Then we fix three values of $E$ such that: the first one corresponds to $E=E_{-}(J)=V_J(r_Q)$, where we know that the solution is a plane wave, the second is strictly between $E_{-}(J)$ and $E_{+}(J)$ and the third is very close to $E_{+}(J)$.
For each value of $E$ we plot the numerical solution of the ODE and the solution of the minimization problem. As we can see in Figure~\ref{fig:ode-min-def-J=0.1}, we obtain a very good agreement with the conjecture. For the first value of $E$ we have a maximum difference between the solutions of $5 \times 10^{-8}$, for the second value of $E$ a  maximum difference of $1 \times 10^{-2}$ and for the third value of $E$ a maximum difference of $3 \times 10^{-2}$.

\begin{figure}[htpb!]
\centering
\includegraphics[width=.45\textwidth]{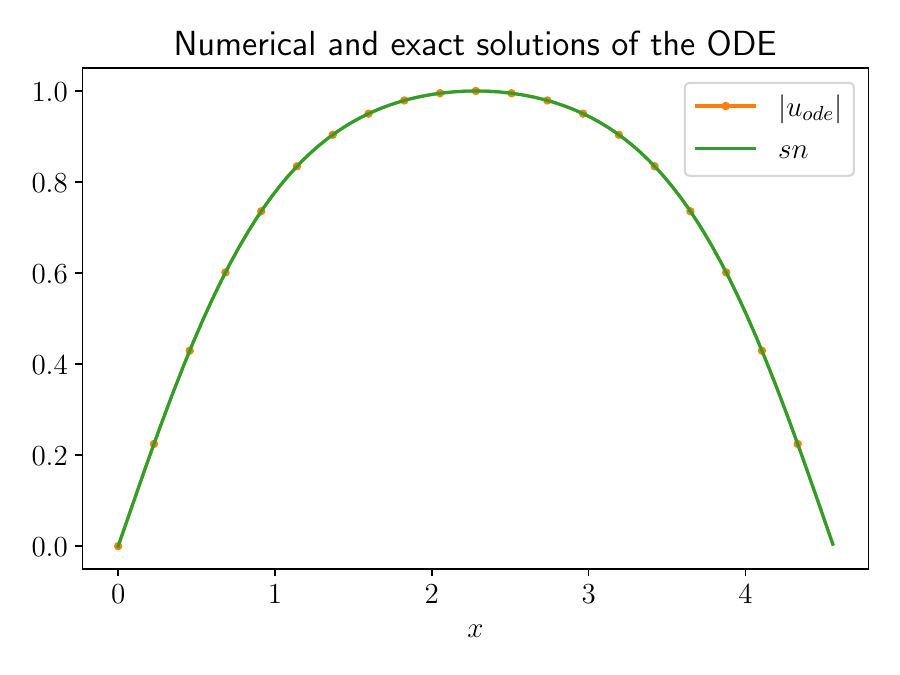}~
\includegraphics[width=.45\textwidth]{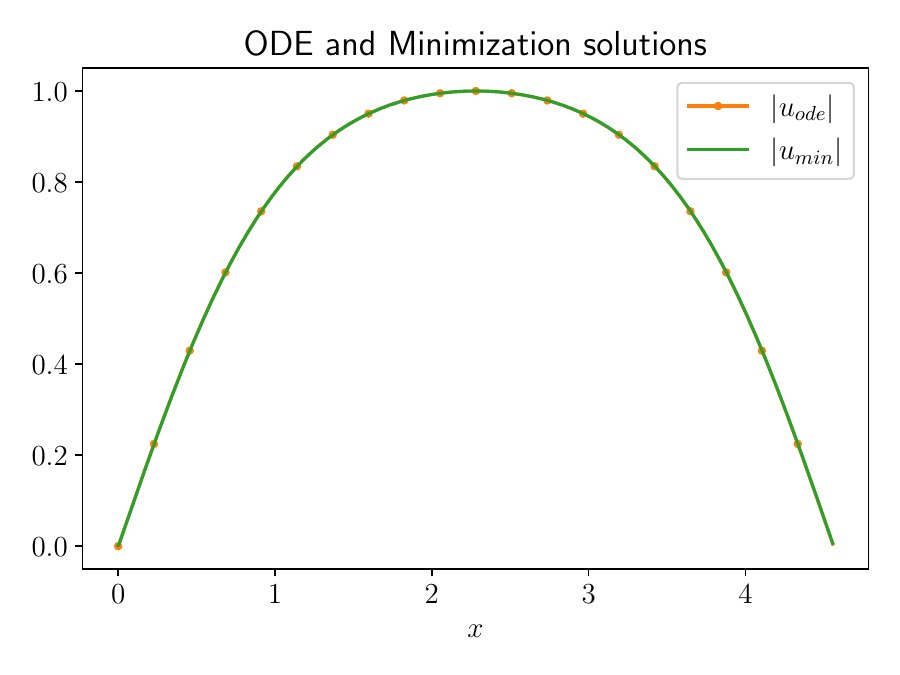}~
\caption{Comparison of the $\sn$ function and the solutions of the ODE and the minimization problem.}
\label{fig:ode-sn}
\end{figure}

\begin{figure}[htpb!]
\centering
\includegraphics[width=.33\textwidth]{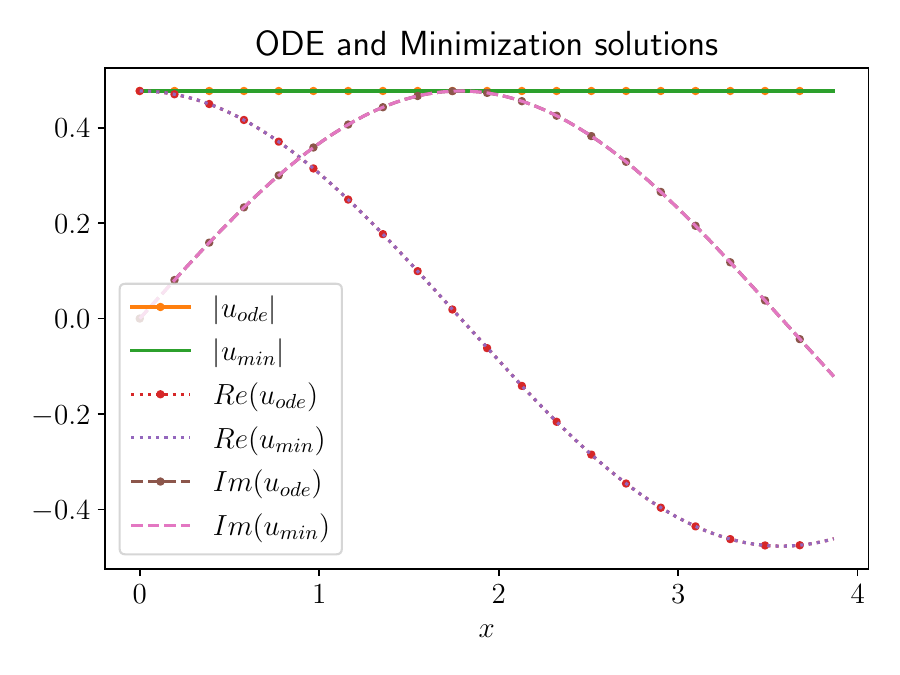}~
\includegraphics[width=.33\textwidth]{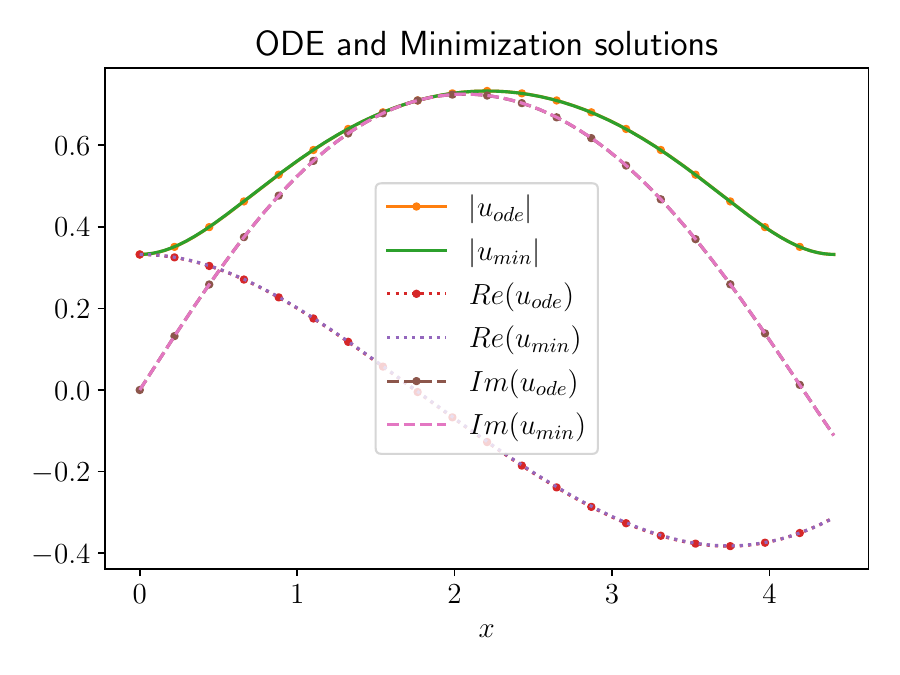}~
\includegraphics[width=.33\textwidth]{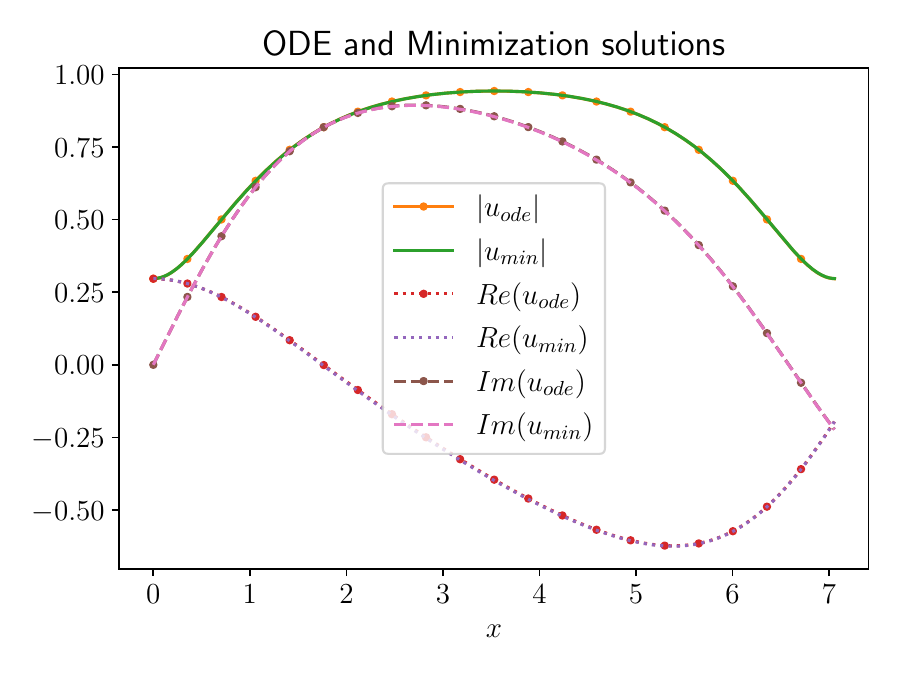}
\caption{Comparison  between $u_{min}$ and $u_{ode}$ for the defocusing case.}
\label{fig:ode-min-def-J=0.1}
\end{figure}

We do the same for the focusing case, with positive $a$. We fix $b=1$ and $a=1$ and an arbitrary $J=1$. We choose two values for $E$. The first one $E=E_{-}(J)=V_J(r_Q)$ and the second such that $E_{-}(J)<E=5$. We plot the solutions and we can see in Figure~\ref{fig:ode-min-foc-a=1-J=1} that the solution of the ODE is the minimizer with a maximum difference of $1.2 \times 10^{-7}$ for the first value of $E$ and $1 \times 10^{-2}$ in the second. 

\begin{figure}[htpb!]
\centering
\includegraphics[width=.5\textwidth]{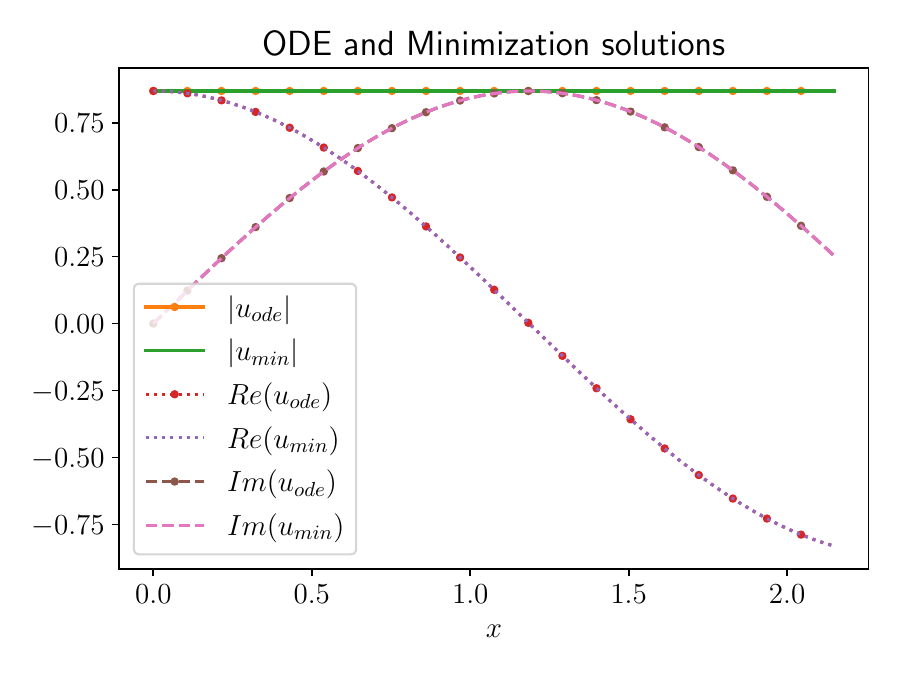}~
\includegraphics[width=.5\textwidth]{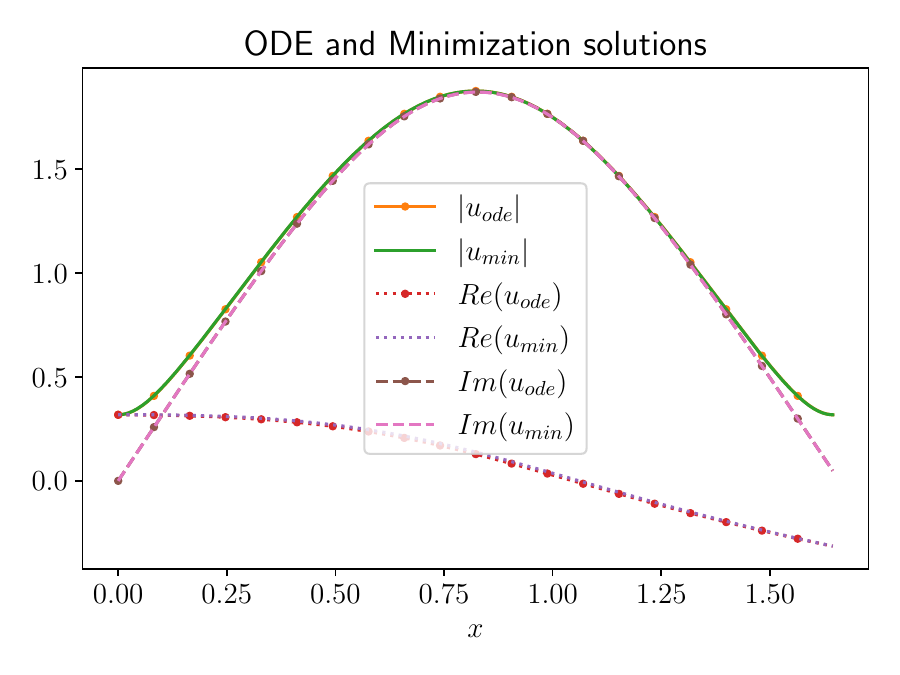}~
\caption{Comparison  between $u_{min}$ and $u_{ode}$ for the focusing case with $a=1$.}
\label{fig:ode-min-foc-a=1-J=1}
\end{figure}

Finally the focusing case with negative $a$. We fix $b=1$ and $a=-1$ and an arbitrary $J=4$. We choose two values for $E$. The first one $E=E_{-}(J)=V_J(r_Q)$ and the second $E_{-}(J)<E=7$. We plot the solutions and we can see in Figure~\ref{fig:ode-min-foc-a=-1-J=4} that the solution of the ODE is the minimizer. For the first value of $E$ we have a maximum difference between the solutions of $6 \times 10^{-8}$, for the second value of $E$ a maximum difference of $2 \times 10^{-3}$.

\begin{figure}[htpb!]
\centering
\includegraphics[width=.5\textwidth]{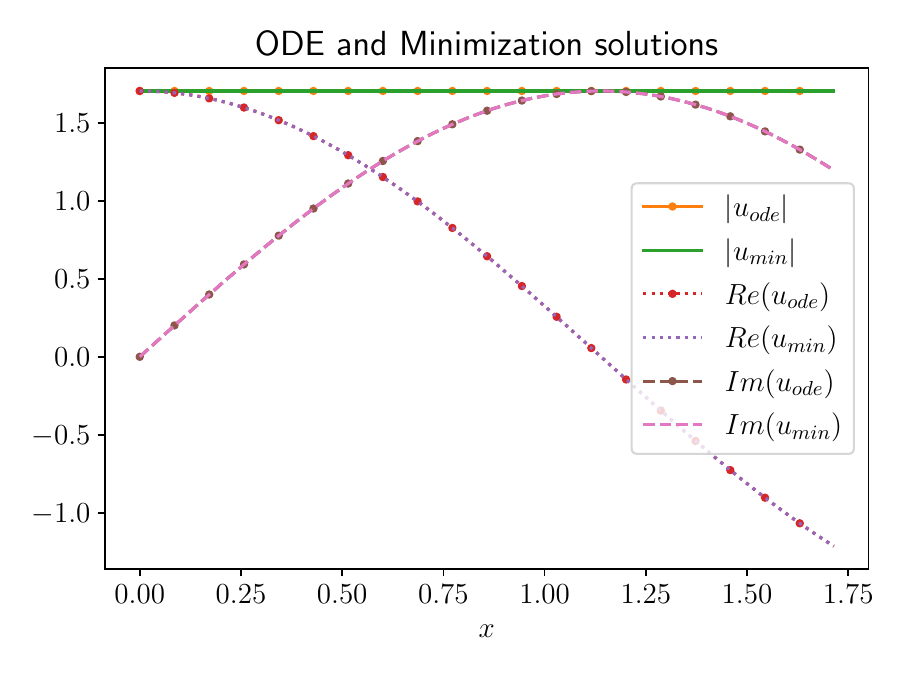}~
\includegraphics[width=.5\textwidth]{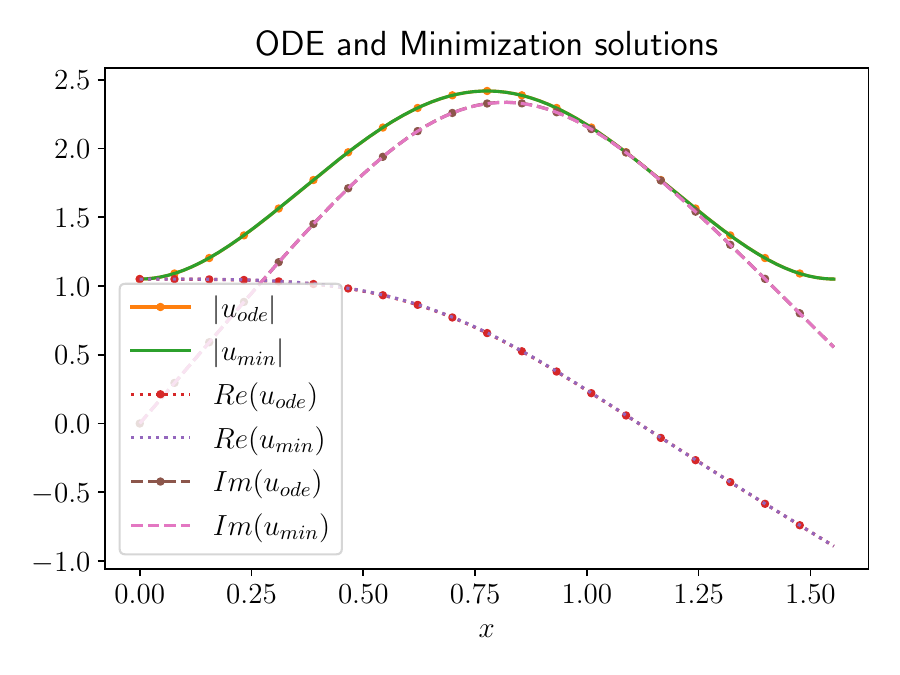}~
\caption{Comparison  between $u_{min}$ and $u_{ode}$ for the focusing case with $a=-1$.}
\label{fig:ode-min-foc-a=-1-J=4}
\end{figure}

\bibliographystyle{abbrv}
\bibliography{bibliography}

\ifx \undefined \booktitle \def \booktitle#1{{{\em #1}}} \fi\ifx \cftil
  \undefined \def \cftil#1{\~#1} \fi\ifx \undefined \cprime \def \cprime {$'$}
  \fi\ifx \undefined \flqq \def \flqq {\ifmmode \ll \else \leavevmode \raise
  0.2ex \hbox{$\scriptscriptstyle \ll $}\fi}\fi\ifx \undefined \frqq \def \frqq
  {\ifmmode \gg \else \leavevmode \raise 0.2ex \hbox{$\scriptscriptstyle \gg
  $}\fi}\fi\ifx \undefined \k \let \k = \c \fi\ifx \undefined \mathbb \def
  \mathbb #1{{\bf #1}}\fi\ifx \undefined \mathbf \def \mathbf #1{{\bf
  #1}}\fi\ifx \undefined \mathrm \def \mathrm #1{{\rm #1}}\fi\ifx \undefined
  \pkg \def \pkg #1{{{\tt #1}}} \fi\ifx \undefined \scr \let \scr = \cal
  \fi\def\cprime{$'$} \ifx \undefined \booktitle \def \booktitle#1{{{\em #1}}}
  \fi\ifx \cftil \undefined \def \cftil#1{\~#1} \fi\ifx \undefined \cprime \def
  \cprime {$'$} \fi\ifx \undefined \flqq \def \flqq {\ifmmode \ll \else
  \leavevmode \raise 0.2ex \hbox{$\scriptscriptstyle \ll $}\fi}\fi\ifx
  \undefined \frqq \def \frqq {\ifmmode \gg \else \leavevmode \raise 0.2ex
  \hbox{$\scriptscriptstyle \gg $}\fi}\fi\ifx \undefined \k \let \k = \c
  \fi\ifx \undefined \mathbb \def \mathbb #1{{\bf #1}}\fi\ifx \undefined
  \mathbf \def \mathbf #1{{\bf #1}}\fi\ifx \undefined \mathrm \def \mathrm
  #1{{\rm #1}}\fi\ifx \undefined \pkg \def \pkg #1{{{\tt #1}}} \fi\ifx
  \undefined \scr \let \scr = \cal \fi\def\cprime{$'$} \def\cprime{$'$}
\begin{thebibliography}{10}

\bibitem{BaDu04}
W.~Bao and Q.~Du.
\newblock Computing the ground state solution of {B}ose-{E}instein condensates
  by a normalized gradient flow.
\newblock {\em SIAM J. Sci. Comput.}, 25(5):1674--1697, 2004.

\bibitem{BeDuLe20}
C.~Besse, R.~Duboscq, and S.~Le~Coz.
\newblock Gradient flow approach to the calculation of stationary states on
  nonlinear quantum graphs.
\newblock {\em Ann. Henri Lebesgue}, 5:387--428, 2022.

\bibitem{FaJe18}
E.~Faou and T.~J\'{e}z\'{e}quel.
\newblock Convergence of a normalized gradient algorithm for computing ground
  states.
\newblock {\em IMA J. Numer. Anal.}, 38(1):360--376, 2018.

\bibitem{Fi15}
G.~Fibich.
\newblock {\em The nonlinear {S}chr\"odinger equation}, volume 192 of {\em
  Applied Mathematical Sciences}.
\newblock Springer, Cham, 2015.

\bibitem{GaHa07-2}
T.~Gallay and M.~H{\v{a}}r{\v{a}}gus.
\newblock Orbital stability of periodic waves for the nonlinear {S}chr\"odinger
  equation.
\newblock {\em J. Dynam. Differential Equations}, 19(4):825--865, 2007.

\bibitem{GaHa07-1}
T.~Gallay and M.~H{\u{a}}r{\u{a}}gu{\c{s}}.
\newblock Stability of small periodic waves for the nonlinear {S}chr\"odinger
  equation.
\newblock {\em J. Differential Equations}, 234(2):544--581, 2007.

\bibitem{GuLeTs17}
S.~Gustafson, S.~Le~Coz, and T.-P. Tsai.
\newblock Stability of periodic waves of 1{D} cubic nonlinear {S}chr\"odinger
  equations.
\newblock {\em Appl. Math. Res. Express. AMRX}, 2:431--487, 2017.

\bibitem{KfLe24}
P.~Kfoury and S.~Le~Coz.
\newblock Variational properties of space-periodic standing waves of nonlinear
  {Schr{\"o}dinger} equations with general nonlinearities.
\newblock {\em ESAIM, Control Optim. Calc. Var.}, 30:31, 2024.
\newblock Id/No 79.

\bibitem{La89}
D.~F. Lawden.
\newblock {\em Elliptic functions and applications}, volume~80 of {\em Applied
  Mathematical Sciences}.
\newblock Springer-Verlag, New York, 1989.

\bibitem{SuSu99}
C.~Sulem and P.-L. Sulem.
\newblock {\em The nonlinear {S}chr\"odinger equation}, volume 139 of {\em
  Applied Mathematical Sciences}.
\newblock Springer-Verlag, New York, 1999.
\newblock Self-focusing and wave collapse.

\end{thebibliography}

\end{document}